\documentclass[]{article}
\usepackage[utf8]{inputenc}
\usepackage[english]{babel}
\usepackage{amsfonts, amssymb, amsthm, mathtools}
\usepackage{graphicx}
\usepackage[colorlinks=true,allcolors=.]{hyperref}
\usepackage{subcaption}
\usepackage[a4paper,margin=1in]{geometry}

\usepackage{csquotes}
\usepackage[backend=biber,style=numeric,maxbibnames=5,giveninits=true]{biblatex}
\usepackage{lmodern}
\usepackage{enumitem}
\usepackage{color}

\usepackage[noabbrev,nameinlink]{cleveref}
\crefname{equation}{}{}

\usepackage{algorithm}
\usepackage{algpseudocode}
\usepackage{algorithmicx}
\usepackage{varwidth}
\usepackage{siunitx}
\usepackage{pgfplots}

\usepackage{todonotes}

\newtheorem{lemma}{Lemma}

\newcommand{\R}{{\mathbb{R}}}
\newcommand{\F}{F}
\renewcommand{\L}{{\mathcal{L}}}

\DeclareMathOperator{\Div}{div}
\DeclareMathOperator{\Tr}{Tr}
\DeclareMathOperator{\id}{id}

\newcommand{\Gobs}{{\Gamma_{\text{obs}}}}
\newcommand{\Gout}{{\Gamma_{\text{out}}}}
\newcommand{\Gin}{{\Gamma_{\text{in}}}}
\newcommand{\Gwall}{{\Gamma_{\text{wall}}}}
\newcommand{\Oobs}{{\Omega_{\text{obs}}}}

\newcommand{\holdall}{{G}}
\newcommand{\bc}{{\text{bc}}}
\newcommand{\vol}{{\text{vol}}}

\newcommand{\test}[1]{{\delta_{#1}}}
\newcommand{\mult}[1]{{\psi_{#1}}}

\newcommand{\control}{{c}}
\newcommand{\eext}{{\eta_{\text{ext}}}}
\newcommand{\edet}{{\eta_{\text{det}}}}
\newcommand{\ppenalty}{\mu}

\newcommand{\ainit}{{\alpha_{\text{init}}}}
\newcommand{\atarget}{{\alpha_{\text{target}}}}
\newcommand{\adec}{{\alpha_{\text{dec}}}}

\title{Mesh quality preserving shape optimization using nonlinear extension operators}
\author{
	Sofiya Onyshkevych\thanks{Department of Mathematics, University~Hamburg, Bundesstr.~55, 20146~Hamburg, Germany, \mbox{(sofiya.onyshkevych@uni-hamburg.de)}}\and 
	Martin Siebenborn\thanks{Department of Mathematics, University~Hamburg, Bundesstr.~55, 20146~Hamburg, Germany, \mbox{(martin.siebenborn@uni-hamburg.de)}}
}

\begin{document}
\maketitle
\begin{abstract}
	In this article, we propose a shape optimization algorithm which is able to handle large deformations while maintaining a high level of mesh quality. Based on the method of mappings we introduce a nonlinear extension operator, which links a boundary control to domain deformations, ensuring admissibility of resulting shapes.
	The major focus is on comparisons between well-established approaches involving linear-elliptic operators for the extension and the effect of additional nonlinear advection on the set of reachable shapes.
	It is moreover discussed how the computational complexity of the proposed algorithm can be reduced.
	The benefit of the nonlinearity in the extension operator is substantiated by several numerical test cases of stationary, incompressible Navier-Stokes flows in 2d and 3d.\\[0.3cm]
\end{abstract}
\section{Introduction}
Shape optimization is a discipline in the field of optimization constrained by partial differential equations (PDEs).
Here the contour of the domain $\Omega$, where typically a PDE models the effects of interest, plays the role of the optimization variable.
Possible variants are that the outer shape of $\Omega$ is to be determined, e.g.\ when $\Omega$ represents a solid body, or interior interfaces, which separate spatially discontinuous coefficients such as material properties.
Shape optimization in general is nowadays an active field of research with applications ranging from electrostatics \cite{langer2015shape}, interface identification in  transmission  processes \cite{schulz2015structured, harbrecht2013numerical, naegel2015scalable}, fluid-dynamics \cite{schmidt2013three, BaLiUl, garcke2016stable}, acoustics \cite{udawalpola2008optimization},  image restoration and segmentation \cite{hintermuller2004second} and composite material identification \cite{siebenborn2017algorithmic, pinzon2020parallel} to nano-optics \cite{hiptmair2018large}.

In this article, we focus on shape optimization in fluid dynamics, which is also one of the pioneering applications in this field \cite{mohammadi2010applied, jameson2003aerodynamic, giles2000introduction}.
In general, the optimization problem can be formulated as
\begin{equation}
\label{eq::general_problem}
\begin{aligned}
&\min\limits_{\Omega \in \holdall_\text{adm}} && j(y, \Omega) \\
&\quad \text{s.t.} && E(y, \Omega) = 0
\end{aligned}
\end{equation}
where $j$ is a shape functional depending on a state variable $y$ and the shape of the domain $\Omega$.
Moreover, $y$ fulfills the PDE constraint $E$, which itself depends on $\Omega$.
A typical example is an obstacle specimen $\Oobs$ in a flow tunnel $\Omega$ as depicted in \cref{fig::domain}.
One of the main questions is an appropriate choice of the set of admissible shapes $\holdall_\text{adm}$, in which optimization takes place.
For problems of this type, two prominent approaches can be identified in the literature.
On the one hand, the Hadamard-Zol\'esio structure theorem is applied, which allows to trace back changes in the objective $j$ solely to variations of the boundary $\Gobs$ (see for instance \cite{sokolowski2012introduction, delfour2001shapes}).
It is thus possible to define directional shape derivatives via variations of $\Gobs$ in a direction normal to the boundary.
Together with the choice of an appropriate shape and tangent space, this allows to represent the sensitivity for $j$ w.r.t.\ $\Gobs$ as a gradient.
This is then interpreted as a deformation to $\Gobs$ and a new discretization mesh for the resulting domain can be computed.
By this step the mesh quality of the deformed domain can be ensured as pursued in, e.g., \cite{Wilke2005,EF16}.
Alternatively, the definition of shape and tangent space includes the surrounding domain $\Omega$, which immediately results in deformation information for the entire mesh (e.g.\ \cite{langer2015shape,schulz2016computational,EHLG18}) and makes the additional call to a mesh generator superfluous.
Typical approaches consider interpreting the shape sensitivity as a force term in linear elastic models described over $\Omega$. The resulting displacement field is then applied as a mesh deformation.
Especially in recent works (see for instance \cite{haubner2020continuous, dokken2019shape, DokkenFunke20}), linear elastic extension equations are considered and, in particular, a very small or even zero first Lam\'e constant is favored.

Moreover, a descent method allows to control the mesh quality from one iteration to the next, i.e.\ for one deformation.
Yet, in the limit of the sequence of design updates, quality is typically lost.
This effect is described in, e.g., \cite{pinzon2020parallel}, where variable interfaces must be prevented from overlapping.

In this article we follow a different approach, which gives a higher level of control on the quality of the mesh around the optimal shape.
Based on the method of mappings (cf.\ to \cite{MuSi}), the question for admissible shapes $\holdall_\text{adm} \coloneqq \lbrace
\F(\Omega)\colon F \in \mathcal{F}_\text{adm}\rbrace$ in \cref{eq::general_problem} is translated to the choice of appropriate function spaces, in which a deformation from reference to the optimal configuration is to be found.
Here $\mathcal{F}_\text{adm}$ denotes a set of admissible mappings.
Starting from a reference configuration $\Omega$, it is then optimized over the transformations $\F(\Omega)$ yet without explicitly performing mesh deformations.
For this purpose the PDE constrained is transformed to the virtual domain as $E(y, \F(\Omega))$.
The optimization problem then turns into a classical optimal control in the form of
\begin{equation}
\label{eq::opt_control_general}
\begin{aligned}
\min\limits_{\F \in \mathcal{F}_\text{adm}} &&& j(y, \F(\Omega)) \\
\text{s.t.}\quad &&& E(y, \F(\Omega)) = 0.
\end{aligned}
\end{equation}
This approach is a recent field of studies and applied in, e.g., \cite{Kunisch2001,Slawig2004,BaLiUl}.
Also based on this approach is the investigation in \cite{haubner2020continuous}, which is the starting point for the consideration in the present article.
Here the problem in \cref{eq::opt_control_general} is formulated as
\begin{equation}
\label{eq::opt_control}
\begin{aligned}
\min\limits_{\control \in L^2(\Gobs)}&&& j(y, \F(\Omega)) + \frac{\alpha}{2} \Vert c \Vert^2_{L^2(\Gobs)} & \\
\quad \text{s.t.}\quad &&& E(y, \F(\Omega)) = 0 & \\
 &&& \F = \id + w & \text{ in } \Omega\\
 &&& \det(D\F) \geq \edet& \text{ in } \Omega\\
 &&& w = S(\control, \Omega) & 
\end{aligned}
\end{equation}
in terms of a regularization parameter $\alpha > 0$ and a bound $\edet > 0$ on the determinant of the derivative of the mapping function $\F$.
The focus of the investigations therein is on the extension operator $S$.
It is suggested to choose $S$ to be the composition of mappings $c \mapsto b \mapsto w$.
Here $c \mapsto b$ is realized via the solution operator of a Laplace-Beltrami equation on $\Gobs$.
The mapping to the actual displacement, i.e.\ $b \mapsto w$, is then chosen to be the solution operator of a vector-valued elliptic equation, such as a linear elastic model.
It is proven that -- under certain circumstances -- the domain mapping $\F$ is locally a $C^1(\bar{\Omega}, \R^d)$-diffeomorphism provided that $\det(D\F) \geq \edet$ is fulfilled.

The main focus of our present article is a numerical study of different choices of the extension operator $S$.
It turns out that optimization settings, where larger deformations are to be expected, are a limiting factor for linear operators $S$.

This limitation is due to the fact that the structure of a shape space, that is as large as possible, can hardly be linear since this would require to explain what scalar multiples or sums of shapes are.
Yet, with the method of mappings and a linear extension operator $S$ we approximate the set of admissible shapes locally by a linear function space of admissible deformations to a reference configuration.

We thus suggest a nonlinear extension mapping and present numerical studies on the applicability.
It should be mentioned that the theory developed so far is not applicable in this case.
It only applies to the linear choice of $S$, which is a special case of the more general consideration in this article.

The motivation for the choice of $S$ in this present work is the observation that, on the one hand, via the condition $\det(D\F) \geq \edet$ the local injectivity of $\F$ can be ensured.
But on the other hand, this limits significantly the subset of admissible transformations $\mathcal{F}_\text{adm}$ and thus affects optimal shapes as outlined in the last section of this article.
It is thus the task to find an operator $S$ which prevents $\det(D\F) \geq \edet$ from becoming active even for large deformations.
We also discuss cases where the reference domain is not of circular shape and demonstrate the performance of the extension and influence on the mesh quality in a deformed domain.
The intention of this experiment is to demonstrate that the set of shapes $\holdall_\text{adm}$, which is constructed via the mappings from $\mathcal{F}_\text{adm}$, can be extended significantly and the dependence on the choice of a reference domain $\Omega$ can be hidden.
In particular, the studies illuminate whether large deformations in the optimization are possible for general reference configurations, which do not fulfill certain properties like convexity or an injective normal vector field.

This article is structured as follows: In \cref{sec::problem_formulation} the shape optimization problem is set up and formulated in terms of the method of mappings. \Cref{sec::algorithm} is devoted to the nonlinear extension model and, furthermore, the derivation of necessary optimality conditions and the presentation of an optimization algorithm.
In \cref{sec::numerics} numerical studies are conducted and discussed.
The article closes in \cref{sec::conclusion} with a conclusion of the results.
\begin{figure}[h!]
	\centering
	\def\svgwidth{0.6\textwidth}
\begingroup%
  \makeatletter%
  \providecommand\color[2][]{%
    \errmessage{(Inkscape) Color is used for the text in Inkscape, but the package 'color.sty' is not loaded}%
    \renewcommand\color[2][]{}%
  }%
  \providecommand\transparent[1]{%
    \errmessage{(Inkscape) Transparency is used (non-zero) for the text in Inkscape, but the package 'transparent.sty' is not loaded}%
    \renewcommand\transparent[1]{}%
  }%
  \providecommand\rotatebox[2]{#2}%
  \newcommand*\fsize{\dimexpr\f@size pt\relax}%
  \newcommand*\lineheight[1]{\fontsize{\fsize}{#1\fsize}\selectfont}%
  \ifx\svgwidth\undefined%
    \setlength{\unitlength}{1750.4442206bp}%
    \ifx\svgscale\undefined%
      \relax%
    \else%
      \setlength{\unitlength}{\unitlength * \real{\svgscale}}%
    \fi%
  \else%
    \setlength{\unitlength}{\svgwidth}%
  \fi%
  \global\let\svgwidth\undefined%
  \global\let\svgscale\undefined%
  \makeatother%
  \begin{picture}(1,0.55442108)%
    \lineheight{1}%
    \setlength\tabcolsep{0pt}%
    \put(0,0){\includegraphics[width=\unitlength,page=1]{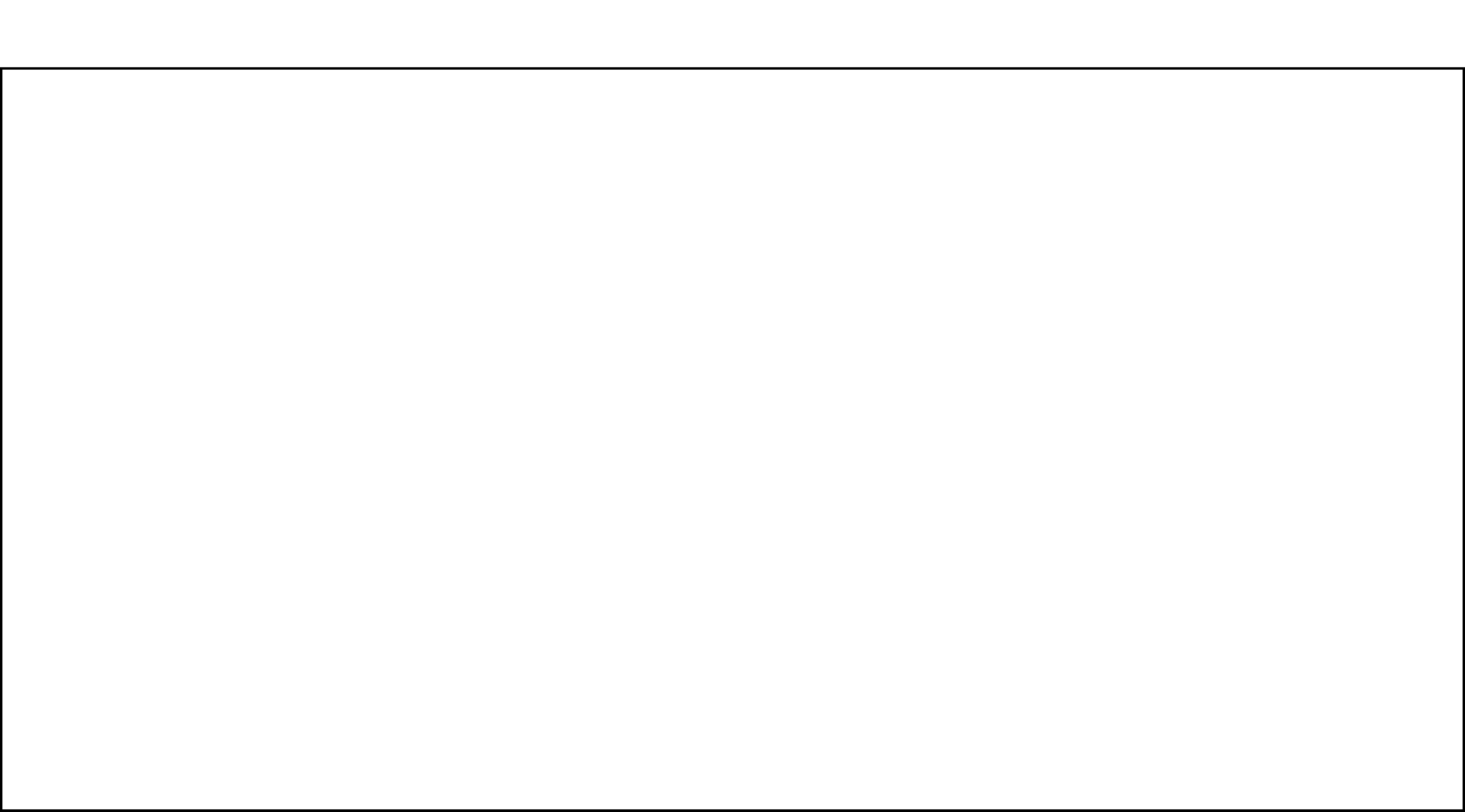}}%
    \put(0.00569204,0.24731643){\color[rgb]{0,0,0}\makebox(0,0)[lt]{\lineheight{0}\smash{\begin{tabular}[t]{l}$\Gin$\end{tabular}}}}%
    \put(0.9901012,0.24731643){\color[rgb]{0,0,0}\makebox(0,0)[rt]{\lineheight{0}\smash{\begin{tabular}[t]{r}$\Gout$\end{tabular}}}}%
    \put(0.45922374,0.00848074){\color[rgb]{0,0,0}\makebox(0,0)[lt]{\lineheight{0}\smash{\begin{tabular}[t]{l}$\Gwall$\end{tabular}}}}%
    \put(0.45922374,0.52056579){\color[rgb]{0,0,0}\makebox(0,0)[lt]{\lineheight{0}\smash{\begin{tabular}[t]{l}$\Gwall$\end{tabular}}}}%
    \put(0,0){\includegraphics[width=\unitlength,page=2]{sketch.pdf}}%
    \put(0.51806556,0.32385346){\makebox(0,0)[lt]{\lineheight{0}\smash{\begin{tabular}[t]{l}$\Gobs$\end{tabular}}}}%
    \put(0.70059944,0.12740709){\makebox(0,0)[lt]{\lineheight{0}\smash{\begin{tabular}[t]{l}$\Omega$\end{tabular}}}}%
    \put(0.39916465,0.2562819){\makebox(0,0)[lt]{\lineheight{0}\smash{\begin{tabular}[t]{l}$\Oobs$\end{tabular}}}}%
  \end{picture}%
\endgroup%

	\caption{Sketch of the holdall domain $\holdall = \Omega \cup \Oobs$.}
	\label{fig::domain}
\end{figure}

\section{Optimization problem}
\label{sec::problem_formulation}
We carry out our considerations based on a classical optimization problem in the field of fluid dynamics described in \cite{mohammadi2010applied}.
In a $d$-dimensional, bounded domain $\Omega$ with Lipschitz boundary, as sketched in \cref{fig::domain}, we consider the minimization of the following energy dissipation functional
\begin{equation}
\label{eq::energy_dissipation}
\min\limits_\Gobs\; j(v, \Gobs) = \frac{\nu}{2} \int_\Omega \sum_{i,j=1}^{d} \left( \frac{\partial v_i}{\partial x_j} \right)^2 \, dx
\end{equation}
where the contour $\Gobs$ of the obstacle $\Oobs$ is assumed to be variable.
The spatial dimension is chosen as $d\in\lbrace 2,3 \rbrace$.
In \cref{eq::energy_dissipation} the velocity field denoted by $v$ is given in terms of the stationary, incompressible Navier-Stokes equations
\begin{equation}
\label{eq::navier_stokes_strong}
\begin{aligned}
-\nu \Delta v + (v\cdot \nabla)v + \nabla p &= 0 & \text{ in } & \Omega\\
\Div v &= 0 & \text{ in } & \Omega\\
v &= v_\infty & \text{ on } & \Gin\\
v &= 0  & \text{ on } & \Gobs \cup \Gwall\\
pn - \nu \frac{\partial v}{\partial n} &= 0  & \text{ on } &\Gout.\\
\end{aligned}
\end{equation}
Together with $\Gobs$ the fluid domain $\Omega$ is allowed to change, but the outer boundaries, i.e.\ $\Gin$, $\Gout$ and $\Gwall$, of the experiment are fixed.
In \cref{eq::navier_stokes_strong} $p$ denotes the pressure, $v_\infty$ describes the velocity profile at the inflow boundary, $n$ is the outer normal vector and $\nu$ the viscosity.
Furthermore, we assume that $\Gobs \cap (\Gin \cup \Gwall \cup \Gout)= \emptyset$ holds during the entire optimization.

For the shape optimization of a specimen $\Oobs$ with respect to functionals of type \cref{eq::energy_dissipation}, it is essential to exclude trivial solutions.
Here, shrinking $\Oobs$ to a point or translations towards $\Gwall$ represents undesired descent directions.
Thus, the optimization problem has to be additionally constrained to geometrical conditions.
Our benchmark problem is to find optimal shapes of a specimen with a given volume, which remains located in the center of the flow tunnel.
This is achieved by fixing barycenter and volume of the obstacle $\Oobs$ with the constraints
\begin{align}
\vol(\Oobs) &= \int_\Oobs 1\, dx =\; \text{const}, \label{eq::geometric_constraints_vol}\\
\bc(\Oobs) &= \frac{1}{\vol(\Oobs)}\int_\Oobs x \, dx =\; \text{const}. \label{eq::geometric_constraints_bc}
\end{align}
Since the computation for the barycenter involves the volume of $\Oobs$ itself, these conditions are coupled in principle.
Yet, if \cref{eq::geometric_constraints_vol} is fulfilled, the term $\vol(\Oobs)^{-1}$ in \cref{eq::geometric_constraints_bc} is constant and can thus be factored out.
By further assuming that the barycenter of the specimen $\Oobs$ is $0\in \R^d$, it is thus sufficient to require $\int_\Oobs x \, dx =\; 0$.

In the following, for a vector-valued function $f: \R^d \to \R^d$, we denote by $Df$ the Jacobian matrix with the ordering $Df  =  \left( \frac{\partial f_i}{\partial x_j}\right)_{i,j=1,\dots, d} \in \R^{d\times d}$.
Let further
\begin{equation}
V := \lbrace v \in H^1(\Omega, \R^d) \colon \Div(v)=0, v\vert_\Gin=v_\infty, v\vert_{\Gwall\cup\Gobs}=0\; \text{a.e.} \rbrace,\; Q \coloneqq \lbrace p \in L^2(\Omega)\colon \int_\Omega p\, dx = 0\rbrace
\end{equation}
and consider the weak formulation of the PDE constraint \cref{eq::navier_stokes_strong}:

Find $(v,p) \in V\times Q$ such that
\begin{equation}
\label{eq::navier_stokes_weak}
\begin{aligned}
\nu\int_{\Omega} D v  : D \test{v}   + (D v\, v)\cdot \test{v} - p \Tr (D \test{v})  \,dx &= 0,\\
-\int_\Omega \test{p} \Tr \left( D v \right) \, dx &= 0
\end{aligned}
\end{equation}
for all test functions $(\test{v}, \test{p}) \in \lbrace \test{v} \in H^1(\Omega) \colon \Div(\test{v})=0,  \test{v}\vert_{\Gwall\cup\Gobs\cup\Gin} = 0 \, \text{a.e.} \rbrace \times Q$.
Note that within this article we are using the symbol $\test{\cdot}$ for test functions associated with a given variable.

In order to reformulate the optimization problem \crefrange{eq::energy_dissipation}{eq::geometric_constraints_bc} as an optimal control problem in appropriate function spaces, we consider from now on the domain $\Omega$ as a fixed reference configuration.
Let $F = \id + w$ with $w \in W^{1, \infty}(\Omega, \R^d)$ such that $F$ results in an admissible deformation for $\Omega$.
For the method of mappings we then consider the state \cref{eq::navier_stokes_weak}, objective \cref{eq::energy_dissipation} and the corresponding state variable $v$ in terms of $\F(\Omega)$.

By means of standard computations we obtain the weak formulation of the optimization problem pulled back to the reference domain $\Omega$ by
\begin{align}
\min\limits_{F \in \mathcal{F}_\text{adm}}\quad& j(v,\F) = \frac{\nu}{2} \int_{\Omega} \left( D v (D\F)^{-1} \right) :  \left( D v (D\F)^{-1} \right)  \det(D\F)\, dx\label{eq::weak_ns_objective}\\
\text{s.t. }\quad& \int_{\Omega} \bigl[ \nu \left( D v (D\F)^{-1} \right) :  \left( D \test{v} (D\F)^{-1} \right)  + (D v (D\F)^{-1}v)\cdot \test{v}\notag\\
&\qquad\qquad - p \Tr \left( D \test{v} (D\F)^{-1} \right) \bigr] \det(D\F)\, dx = 0,\label{eq::weak_ns_constraints_1}\\
-& \int_\Omega \test{p} \Tr (D v(D\F)^{-1}) \det(D\F) \,dx = 0,\label{eq::weak_ns_constraints_2}\\
& \int_{\Oobs} \det(D\F) - 1 \, dx = 0,\label{eq::weak_ns_constraints_3}\\
& \int_{\Oobs} F \,  \det(D\F) \, dx = 0 \label{eq::weak_ns_constraints_4}
\end{align}
for all test functions $(\test{v}, \test{p}) \in V\times Q$.
The optimization problem \crefrange{eq::weak_ns_objective}{eq::weak_ns_constraints_4} still leaves open the question for the set of admissible mappings $\mathcal{F}_\text{adm}$.
We thus follow the same approach as in \cite{haubner2020continuous} and translate it into the form of \cref{eq::opt_control}.
By reformulating the constraint $\det(D\F) \geq \edet$ as a penalty term, we obtain the final optimal control problem
\begin{equation}
\begin{aligned}
\label{eq::boundary_control_problem}
\min\limits_{\control \in L^2(\Gobs)} \;&J(v,\control) \coloneqq j(v,\F) + \frac{\alpha}{2} \int_\Gobs  c^2\, ds + \frac{\beta}{2} \int_\Omega \left((\edet - \det(D\F))_+\right)^2\, dx\\[0.2cm]
\text{s.t. }\quad&\text{\crefrange{eq::weak_ns_constraints_1}{eq::weak_ns_constraints_4}}\\
&\F = \id + w\\
&w = S(c),\\
\end{aligned}
\end{equation}
where $(\cdot)_+$ denotes the positive-part function.
The missing piece is now mapping from a scalar-valued boundary control $\control$ to admissible deformation fields $w$, which is the subject of the next section.

\section{Nonlinear extension operators}
\label{sec::algorithm}
Consider the optimal control problem \cref{eq::boundary_control_problem}.
The core of the reformulated shape optimization is the choice of the extension operator $S$, which links a scalar-valued boundary control $\control$ living on $\Gobs$ to a vector-valued displacement field $w$ in $\Omega$.
A domain transformation mapping $\F = \id + w$ is then obtained by the so-called perturbation of identity.
In particular, $w$ has to fulfill certain regularity properties as investigated in \cite{haubner2020continuous}.
It yet turns out in \cref{sec::numerics} that for large deformations, i.e.\ when the reference domain and the optimal configuration differ significantly, linear operators $S$ do not lead to satisfying results.
Note that the choice of $S$ significantly influences the set of reachable shapes $\holdall_\text{adm}$ determined via $\mathcal{F}_\text{adm}$.
It is thus our intention to find $S$ which allows for large deformations without significantly restricting $\holdall_\text{adm}$.
Simultaneously, the corresponding mesh deformations $\F(\Omega)$ should exhibit high element qualities for further usage in numerical simulations. 

The focus of the present article is thus to propose and study nonlinear extensions $S$ given in terms of the solution operator of the coupled PDEs
\begin{equation}
\begin{aligned}
b - \Delta_\Gobs b &= cn & \text{ on } & \Gobs\\
-\Div(\nabla w + \nabla w^\top) + \eext(w \cdot \nabla)w &=0 & \text{ in } &\Omega\\
(\nabla w + \nabla w ^\top) \cdot n &= b& \text{ on } &\Gobs\\
w &=0 & \text{ on } &\Gwall\cup\Gin\cup\Gout
\end{aligned}
\label{eq::nonlinear_extension_strong}
\end{equation}
In the equation above $\Delta_\Gobs$ denotes the vector-valued Laplace-Beltrami operator.
Note that by solving \cref{eq::nonlinear_extension_strong} the scalar-valued control $\control$ is mapped to a vector valued quantity $b$.
The benefit of this particular extension operator, and especially the nonlinearity $\eext(w \cdot \nabla)w$, which is in the focus of this article, becomes particularly visible for experiments with large deformations as pointed out in \cref{subsec::factor-influence}.
A popular choice, as discussed in the introduction, is to define the extension only via the linear term $\Div(\nabla w + \nabla w^\top)$.
Yet, this restricts the set $\mathcal{F}_\text{adm}$ significantly.
This is visible especially for problems in fluid dynamics, as pointed out in \cref{sec::numerics}, where the reference shape is of spherical type but the optimum to be found is stretched and approximates non-smooth tip and back.

Problems arise due to strong compressions of finite elements in the discretization orthogonal to the main deformation direction.
This observation motivates to add the nonlinear advection term $\eext(w \cdot \nabla)w$, which -- geometrically speaking -- promotes displacements $w$ where nodes move along large gradients.
This results in a homogeneous distribution of finite elements even around approximately non-smooth regions of $\Gobs$. 

For \cref{eq::nonlinear_extension_strong}, which specifies the mapping from boundary control to domain deformation, the weak formulation is given by: Find $w \in H_0^1(\Omega, \R^d)$ such that
\begin{align}
\int_\Gobs b\cdot \test{b} + D_{\Gobs} b : D_{\Gobs} \test{b} \, ds &= \int_\Gobs \control n \cdot \test{b} \, ds \label{eq::nonlinear_ext_lb} \\
\int_{\Omega} (Dw + Dw^\top):D\test{w} + \eext(Dw\, w)\cdot \test{w} \, dx &= \int_\Gobs b \test{w} \, ds \label{eq::nonlinear_ext}
\end{align}
for all $\test{w} \in H_0^1(\Omega, \R^d)$ and in terms of $\eext \geq 0$.
Here $D_\Gobs$ denotes the derivative tangential to $\Gobs$.
In \cref{eq::nonlinear_ext_lb} the scalar-valued boundary control $\control \in L^2(\Gobs)$ is multiplied with the outer normal vector field $n$ to $\Omega$ at $\Gobs$.
Then a vector-valued Laplace-Beltrami equation is solved over $\Gobs$.
This is coupled with nonlinear \cref{eq::nonlinear_ext} where the influence of the advection term is controlled via $\eext$.
Note that the linear extension operators investigated in \cite{haubner2020continuous} arise as a special case of system \cref{eq::nonlinear_ext_lb,eq::nonlinear_ext}.

Now that the extension operator is chosen, we can combine the optimization problem \cref{eq::boundary_control_problem} with the extension operator  \cref{eq::nonlinear_ext_lb,eq::nonlinear_ext} to obtain the Lagrangian
\begin{multline}
\label{eq::Lagrangian}
\L ( w, v, p, b, \control, \mult{w},\mult{v},\mult{p}, \mult{b}, \mult{\vol}, \mult{\bc}) =\\
\begin{aligned}
&\frac{\nu}{2} \int_{\Omega} \left( D v (D\F)^{-1} \right) :  \left( D v (D\F)^{-1} \right)  \det(D\F)\, dx +\frac{\alpha}{2}\int_\Gobs \control^2 \, ds
+ \frac{\beta}{2} \int_{\Omega} (( \edet - \det(D\F) )_+ )^2 \, dx \\
&- \int_{\Omega} \bigl[ \nu \left( D v (D\F)^{-1} \right) :  \left( D \mult{v} (D\F)^{-1} \right)  + (D v (D\F)^{-1}v)\cdot \mult{v} - p \Tr \left( D \mult{v} (D\F)^{-1} \right) \bigr] \det(D\F)\, dx\\
&+ \int_\Omega \mult{p} \Tr (D v(D\F)^{-1}) \det(D\F) \,dx\\
&- \int_\Omega (D w + D w^\top): D \mult{w} + \eext(Dw\, w)\cdot \mult{w} \, dx + \int_\Gobs b\cdot \mult{w} \, ds\\
&- \int_\Gobs b\cdot \mult{b} + D_{\Gobs} b : D_{\Gobs} \mult{b} \, ds + \int_\Gobs \control n \cdot \mult{b} \, ds \\
&- \mult{\bc} \cdot \int_{\Omega} (x + w) \,  \det(D\F) \, dx - \mult{\vol} \int_{\Omega} \det(D\F) - 1 \, dx,
\end{aligned}
\end{multline}
where $\mult{\cdot}$ denotes for each variable the associated multiplier.
Note that there is no variable corresponding to the multipliers $\mult{\vol} \in \R$ and $\mult{\bc} \in \R^d$.
These are the finite dimensional multipliers for the barycenter and volume condition \cref{eq::geometric_constraints_vol,eq::geometric_constraints_bc}.
\begin{lemma}
\label{lem::opt_sys}
The first order optimality system associated to the Lagrangian $\L$ in \cref{eq::Lagrangian} is given by the derivatives $\L_{w}, \L_{v}, \L_{p}, \L_{b}, \L_{\mult{w}}, \L_{\mult{v}}, \L_{\mult{p}}, \L_{\mult{b}}, \L_{\control}, \L_{\mult{\vol}}, \L_{\mult{\bc}}$ as 
\begingroup
\allowdisplaybreaks
\begin{align}
	\L_{w}  \test{w} = &
	- \nu\int_\Omega ( Dv (D \F)^{-1}): (D v (D \F)^{-1} D \test{w} (D \F)^{-1})  \det(D\F) \, dx \notag\\&
	+ \frac{\nu}{2} \int_\Omega (Dv (D \F)^{-1}): (Dv (D \F)^{-1}) \Tr((D\F)^{-1} D \test{w}) \det(D\F) \, dx \notag\\&
	- \beta \int_\Omega (\edet - \det(D\F))_+ \Tr((D\F)^{-1} D \test{w}) \det(D\F) \, dx\notag\\&
	+ \nu\int_\Omega (Dv (D \F)^{-1} D \test{w} (D \F)^{-1}) : ( D \mult{v} (D \F)^{-1}) \det(D\F) \, dx \notag\\& 
	+ \nu\int_\Omega (Dv (D \F)^{-1}): ( D \mult{v} (D \F)^{-1} D \test{w} (D \F)^{-1}) \det(D\F) \, dx \notag\\& 
	- \nu\int_\Omega (Dv (D \F)^{-1}) : ( D \mult{v} (D \F)^{-1})\Tr((D\F)^{-1} D \test{w}) \det(D\F) \, dx \notag\\&
	+ \int_\Omega (D v (D \F)^{-1} D \test{w} (D \F)^{-1}\, v)\cdot \mult{v} \det(D\F) \, dx \notag\\&
	- \int_\Omega (D v (D \F)^{-1} \, v)\cdot \mult{v} \Tr((D\F)^{-1} D \test{w}) \det(D\F) \, dx \label{eq::opt_sys_w}\\&
	- \int_\Omega p \Tr( D \mult{v} (D \F)^{-1} D \test{w} (D \F)^{-1}) \det(D\F) \, dx \notag\\&
	+ \int_\Omega p \Tr( D \mult{v} (D \F)^{-1}) \Tr((D\F)^{-1} D \test{w}) \det(D\F) \, dx \notag\\&
	+ \int_\Omega \mult{p} \Tr(D v (D \F)^{-1} D \test{w} (D \F)^{-1}) \det(D\F) \, dx \notag\\&
	- \int_\Omega \mult{p} \Tr( D v (D \F)^{-1}) \Tr((D\F)^{-1} D \test{w}) \det(D\F) \, dx \notag\\&
	- \int_\Omega ( D \test{w} + D \test{w}^\top) : D \mult{w} + \eext((D\test{w}\, w)+(Dw\, \test{w}))\cdot \mult{w}\, dx \notag\\&
	+\beta \int_\Omega (\edet - \det(D\F))_+ \Tr((D\F)^{-1} D \test{w}) \det(D\F) \, dx\notag\\&
	- \mult{\bc} \cdot \int_\Omega \test{w} \det(D\F) + (x + w) \Tr((D\F)^{-1} D \test{w}) \det(D\F) \, dx\notag\\&
	- \mult{\vol} \int_\Omega \Tr((D\F)^{-1} D \test{w}) \det(D\F) \, dx = 0\notag,
\end{align}
\endgroup
\begin{equation}
\L_{\mult{w}}  \test{\mult{w}} =
- \int_\Omega (D w + D w^\top): D \test{\mult{w}} + \eext(Dw\, w)\cdot \test{\mult{w}} \, dx + \int_\Gobs b\cdot \test{\mult{w}} \, ds = 0,
\label{eq::opt_sys_adj_w}
\end{equation}
\begin{align}
\begin{split}
\L_{v}  \test{v} = & \nu\int_{\Omega} \left( D \test{v} (D\F)^{-1} \right) :  \left( D v (D\F)^{-1} \right)  \det(D\F)\, dx\\
&- \nu\int_{\Omega}  \left( D \test{v} (D\F)^{-1} \right) :  \left( D \mult{v} (D\F)^{-1} \right)  \det(D\F) \,dx\\
&-\int_{\Omega} (D \test{v} (D\F)^{-1}v)\cdot \mult{v} +(D v (D\F)^{-1}\test{v})\cdot \mult{v}\det(D\F)\, dx\\
&- \int_{\Omega} \test{p} \Tr (D \test{v}(D\F)^{-1}) \det(D\F) \,dx = 0,
\end{split}
\label{eq::opt_sys_v}
\end{align}
\begin{equation}
\label{eq::opt_sys_adj_v}
\begin{aligned}
\L_{\mult{v}} \test{\mult{v}} = &
-\nu\int_{\Omega}  \left( D v (D\F)^{-1} \right) :  \left( D \test{\mult{v}} (D\F)^{-1} \right)  \det(D\F) \,dx \\
&-\int_{\Omega} (D v (D\F)^{-1}v)\cdot \test{\mult{v}} \det(D\F)\, dx\\
&+ \int_{\Omega} p \Tr (D \test{\mult{v}}(D\F)^{-1}) \det(D\F) \,dx = 0,
\end{aligned}
\end{equation}
\begin{equation}
	\label{eq::opt_sys_p}
	\L_{p}  \test{p} = - \int_\Omega \test{p} \Tr \left( D \mult{v} (D\F)^{-1} \right) \det(D\F)\, dx = 0,
\end{equation}
\begin{equation}
	\label{eq::opt_sys_adj_p}
	\L_{\mult{p}}  \test{\mult{p}} = \int_\Omega \test{\mult{p}} \Tr \left( D v (D\F)^{-1} \right) \det(D\F)\, dx = 0,
\end{equation}
\begin{equation}
    \label{eq::opt_sys_b}
	\L_{b}  \test{b} =
	- \int_\Gobs \test{b}\cdot \mult{b} + D_{\Gobs} \test{b} : D_{\Gobs} \mult{b} \, ds + \int_\Gobs \test{b} \cdot \mult{w} \, ds = 0,
\end{equation}
\begin{equation}
    \label{eq::opt_sys_adj_b}
	\L_{\mult{b}}  \test{\mult{b}} =
	- \int_\Gobs b\cdot \test{\mult{b}} + D_{\Gobs} b : D_{\Gobs} \test{\mult{b}} \, ds + \int_\Gobs \control n \cdot \test{\mult{b}} \, ds = 0,
\end{equation}
\begin{equation}
    \label{eq::opt_sys_c}
	\L_{\control}  \test{\control} =
	\alpha \int_\Gobs \control \test{\control} \, ds + \int_\Gobs \test{\control} n \cdot b \, ds = 0,
\end{equation}
\begin{equation}
    \label{eq::opt_sys_adj_vol}
	\L_{\mult{\vol}}  \test{\mult{\vol}} = - \test{\mult{\vol}} \int_{\Omega} \det(D\F) - 1 \, dx = 0,
\end{equation}
\begin{equation}
	\label{eq::opt_sys_adj_bc}
	\L_{\mult{\bc}}  \test{\mult{\bc}} =
	- \test{\mult{\bc}} \cdot \int_{\Omega} (x + w) \,  \det(D\F) \, dx = 0,
\end{equation}
for all test functions $\test{w}$, $\test{v}$, $\test{p}$, $\test{b}$, $\test{\mult{w}}$, $\test{\mult{v}}$, $\test{\mult{p}}$, $\test{\mult{b}}$, $\test{\control}$, $\test{\mult{\vol}}$, and $\test{\mult{\bc}}$.
\end{lemma}
\begin{proof}
The derivatives of $\L$ are obtained by utilizing standard rules of differentiation.
Note that we particularly use the following identities
\begin{equation*}
\frac{\partial \det(D\F)}{\partial w} \test{w} = \Tr((D\F)^{-1} \test{w})\det(D\F) \quad\text{ and }\quad \frac{\partial (D\F)^{-1}}{\partial w} \test{w} = - (D\F)^{-1} D\test{w} (D\F)^{-1}.
\end{equation*}
For the derivative of the penalty term we utilize that
\begin{equation*}
\begin{aligned}
\frac{\partial}{\partial w}(( \edet - \det(D\F) )_+ )^2 (w)\, \test{w} &= 2 ( \edet - \det(D\F) )_+ \, \chi_{\lbrace\edet > \det(D\F)\rbrace}\, \frac{\partial}{\partial w} \det(D\F)(w) \test{w}\\ &= 2 ( \edet - \det(D\F) )_+ \, \Tr((D\F)^{-1} \test{w})\det(D\F).
\end{aligned}
\end{equation*}
\end{proof}

Recall that the condition $\det(D\F) \geq \edet$ in the problem formulated in \cref{eq::opt_control} is realized via the penalty term $\frac{\beta}{2} \int_{\Omega} (( \edet - \det(D\F) )_+ )^2 \, dx$.
The corresponding term in \cref{eq::opt_sys_w} of the optimality system in \cref{lem::opt_sys} is non-differentiable due to the positive-part function $(\cdot)_+$.
Following the discussions in \cite[sec.\ 3.5]{haubner2020continuous} the mapping $w \mapsto -\beta \int_\Omega (\edet - \det(D\F))_+ \Tr((D\F)^{-1} D \test{w}) \det(D\F) \, dx$ is semismooth and one can compute an element from the generalized derivative in direction $\delta^\prime$ as
\begin{multline*}
(\test{w}, \delta^\prime) \mapsto \beta \int_\Omega \chi_{\lbrace\edet > \det(D\F)\rbrace} \Tr((D\F)^{-1} D \delta^\prime)  \Tr((D\F)^{-1} D \delta^\prime) \det(D\F)^2\\
+ (\edet - \det(D\F))_+ \Tr((D\F)^{-1} D\delta^\prime (D\F)^{-1} D \test{w}) \det(D\F)\\
- (\edet - \det(D\F))_+ \Tr((D\F)^{-1} D \test{w}) \Tr((D\F)^{-1} D\delta^\prime) \det(D\F)\, dx.
\end{multline*}

In the following we briefly present a solution algorithm for the optimality system \crefrange{eq::opt_sys_w}{eq::opt_sys_adj_bc}.
For this purpose we pursue a similar approach as in \cite{haubner2020continuous}.
The core of this method is to solve the nonlinear shape optimization problem \cref{eq::boundary_control_problem} for a decreasing sequence of regularization parameters $\alpha_k$, starting from $\alpha_0 = \ainit$ until the desired level $\atarget$ is reached.
Because each subsequent optimization problem $k$ is nonlinear, this approach benefits from utilizing the known values $y_k \coloneqq (w, v, p, b, \control, \mult{w},\mult{v},\mult{p}, \mult{b}, \mult{\vol}, \mult{\bc})_{k}$ as initial guess in the $(k+1)$-th iteration.
\Cref{alg::direct} summarizes this procedure.
Since parts of the optimality system are non-differentiable, we apply a semismooth Newton's method.
\begin{algorithm}[ht!]
	\caption{Direct optimization algorithm}
	\begin{algorithmic}[1]
		\Require $ 0 < \atarget \leq \ainit$, $0 < \adec < 1$
		\State Set $y_{0}$ to zero
		\State $k \gets 0$
		\State $\alpha_k \gets \ainit$
		\While{$\alpha_k \geq \atarget$}
		\State \parbox[t]{0.8\textwidth}{ Solve \crefrange{eq::opt_sys_w}{eq::opt_sys_adj_bc} for  $y_{k+1}$ with semismooth Newton's method,
			$y_{k}$ as initial guess and regularization parameter $\alpha_k$\strut}
		\State $\alpha_{k+1} \gets \adec\alpha_{k}$
		\State $k \gets k+1$
		\EndWhile
	\end{algorithmic}
	\label{alg::direct}
\end{algorithm}
In \cref{sec::numerics} we demonstrate how to choose the paramter $\ainit, \adec$ and $\atarget$ and illustrate their influence.

\section{Numerical results}
\label{sec::numerics}
This section is devoted to different numerical case studies of stationary, incompressible Navier-Stokes shape optimization problems.
The purpose is to illuminate features of the nonlinear extension operator $S$ proposed in \cref{sec::algorithm}.
In particular, the benefit for optimization benchmark problems, which involve large deformations from the reference configuration to optimal shapes, is numerically investigated.
It is moreover discussed how the local injectivity can be extended to globally injective transformation mappings by adding an artificial volume to the aerodynamic specimen.
Furthermore, algorithmic solvability of the optimality system \crefrange{eq::opt_sys_w}{eq::opt_sys_adj_bc} is addressed in the end of this section.

The experimental settings for the tests are chosen to be comparable in 2d and 3d, respectively.
The holdall domain $\holdall \in \lbrace \holdall_\text{2d}, \holdall_\text{3d} \rbrace$, which reflects the flow tunnel in the experiment, is chosen as \begin{equation*}
\holdall_\text{2d} \coloneqq \lbrack -7, 7 \rbrack \times \lbrack -3, 3 \rbrack \;\text{ and }\; \holdall_\text{3d} \coloneqq \lbrace x \in \R^3: -7 \leq x_1 \leq 7, \sqrt{x_2^2 + x_3^2} \leq 3 \rbrace.
\end{equation*}
Let $\delta$ denote the diameter of the flow tunnel $G$.
We then fix the velocity at the inflow boundary $\Gin$ by $v_\infty = \left( \cos(\frac{2 \pi \Vert x \Vert_2}{\delta}), 0, \dots, 0\right)^\top \in \R^d$.
In all experiments where the specimen $\Oobs$ is a circle or sphere, the radius is given by $r = \num{0.5}$ and $\bc(\Oobs) = 0\in \R^d$.

The discretization of all appearing PDEs is carried out with standard, piecewise linear P1 finite elements.
In order to guarantee stability, we follow the pressure stabilized Petrov Galerkin approach (see for instance \cite{Hughes1986new}), which utilizes an additional term for the pressure $p$ and its adjoint variable $\mult{p}$.
The system under consideration is thus enriched by the two equations
\begin{align*}
g_p &\coloneqq \ppenalty\sum_{T \in \mathcal{T}_h} h_T^2 \int_T \left((D\F)^{-1}\nabla p\right) \cdot \left((D\F)^{-1}\nabla \test{p} \right)\, dx\\
g_\mult{p} &\coloneqq \ppenalty\sum_{T \in \mathcal{T}_h} h_T^2 \int_T \left((D\F)^{-1}\nabla \mult{p}\right) \cdot \left((D\F)^{-1}\nabla \test{\mult{p}} \right)\, dx
\end{align*}
where $\mathcal{T}_h$ denotes the set of all finite elements and $h_T$ measures the longest edge of element $T$.
For each of the subsequent experiments, $\ppenalty = \num{0.1}$ is chosen.

All computations related to the finite element method are carried out using the GETFEM++ library \cite{GETFEM}.
We utilize a parallel version of the library, which relies on PARMETIS \cite{PARMETIS} for mesh partitioning and load balancing.
All linear systems are handled via the parallel factorization solver MUMPS \cite{MUMPS}.
Both, 2d and 3d discretization meshes are produced with the GMSH toolbox \cite{GMSH} and the Delaunay algorithms therein.
If not stated otherwise, all 2d experiments follow the strategy of \cref{alg::direct} with the choice $\ainit=\num{1e-4}$, $\adec=\num{1e-1}$ and $\atarget=\num{1e-10}$.

\subsection{Non-convex shapes with large deformations}
\begin{figure}[h!]
	\centering
	\includegraphics[width=0.45\textwidth]{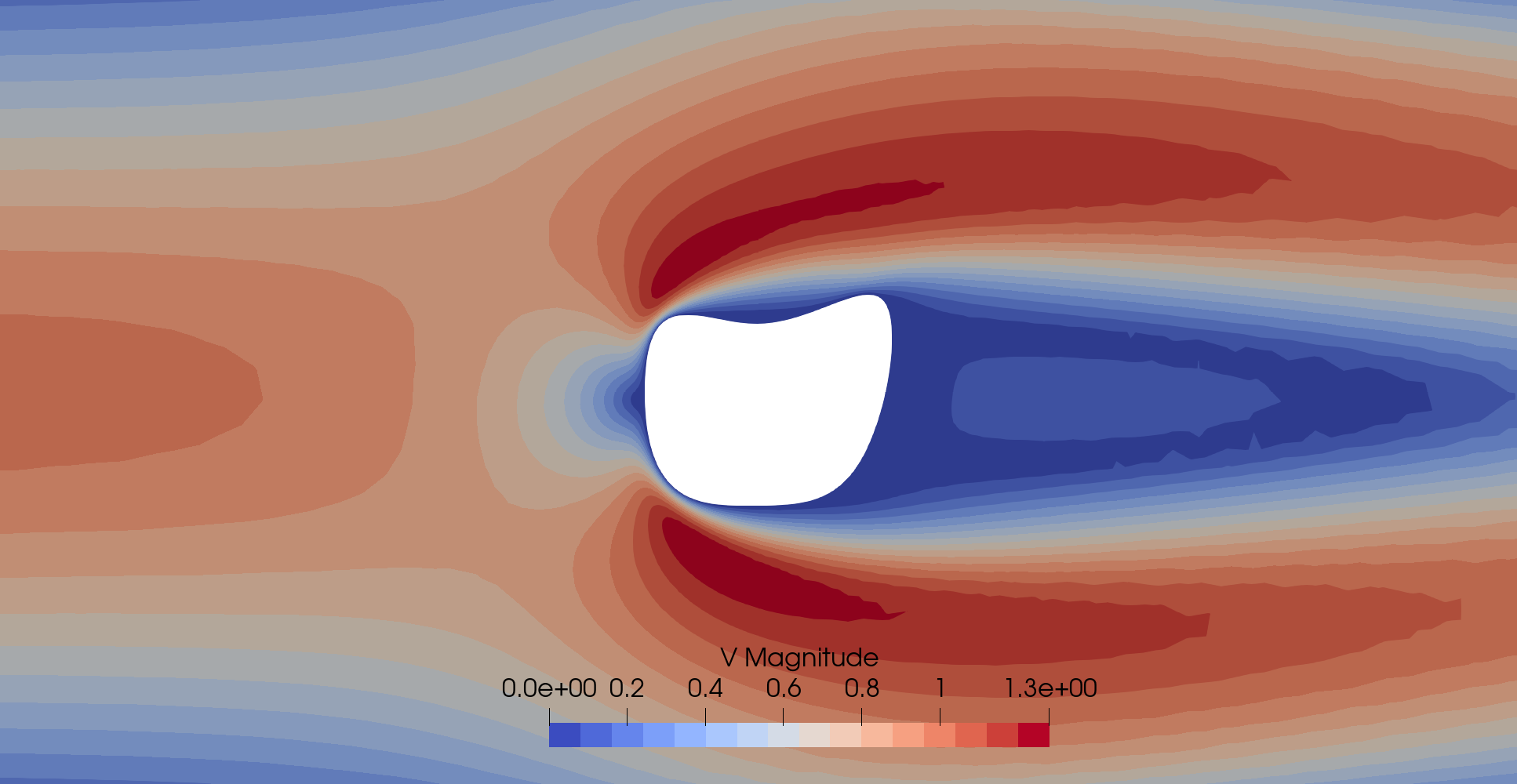}
	\includegraphics[width=0.45\textwidth]{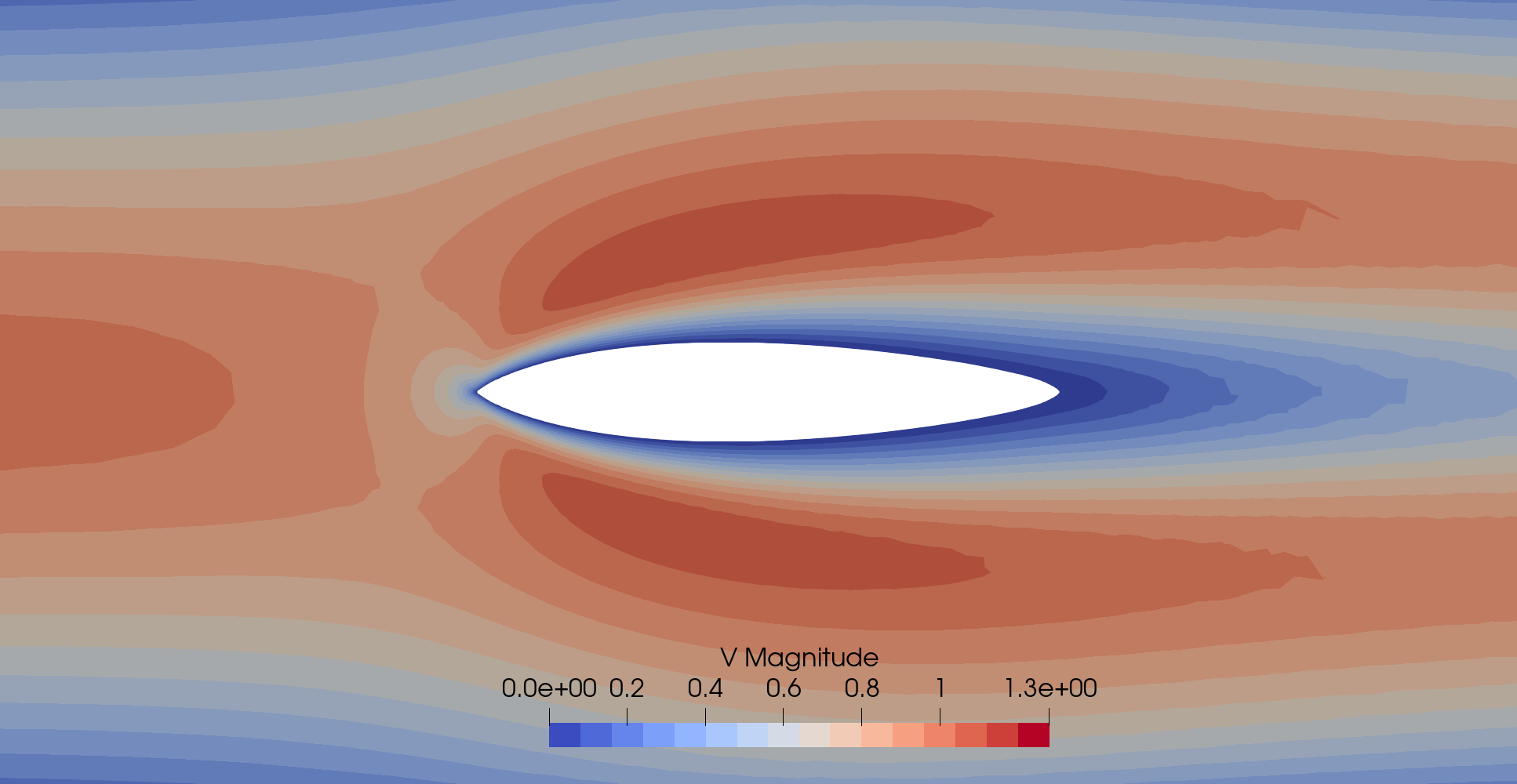}
	\caption{Magnitude of velocity $v$ computed on reference domain $\Omega$ (left) and deformed $\F(\Omega)$ (right).}
	\label{fig::spline-velocity}
\end{figure}
\begin{figure}[h!]
	\centering
	\begin{subfigure}{0.7\textwidth}
		\includegraphics[width=1.0\textwidth]{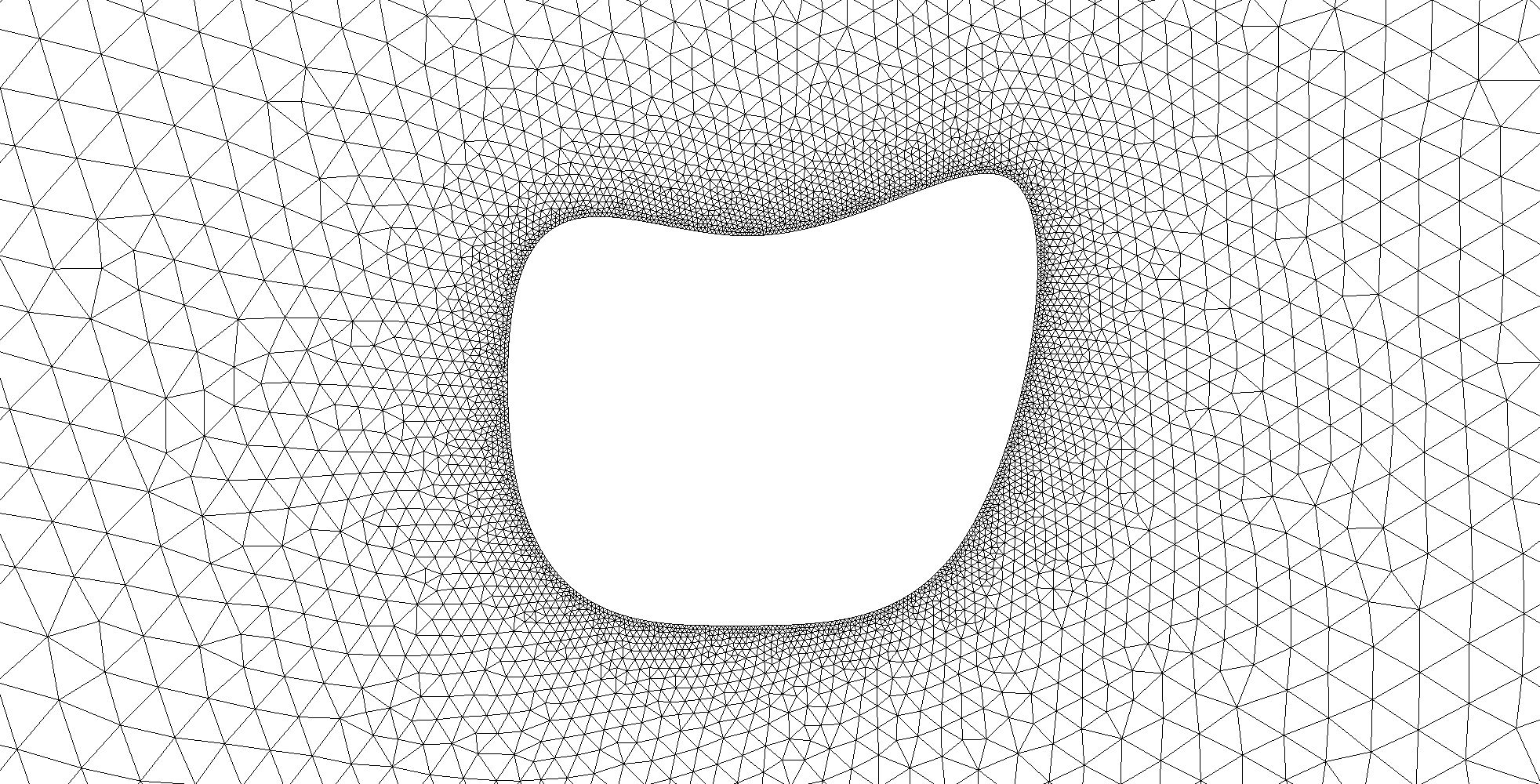}
		\subcaption{Reference grid}
	\end{subfigure}\\[0.4cm]
	\begin{subfigure}{0.7\textwidth}
		\includegraphics[width=1.0\textwidth]{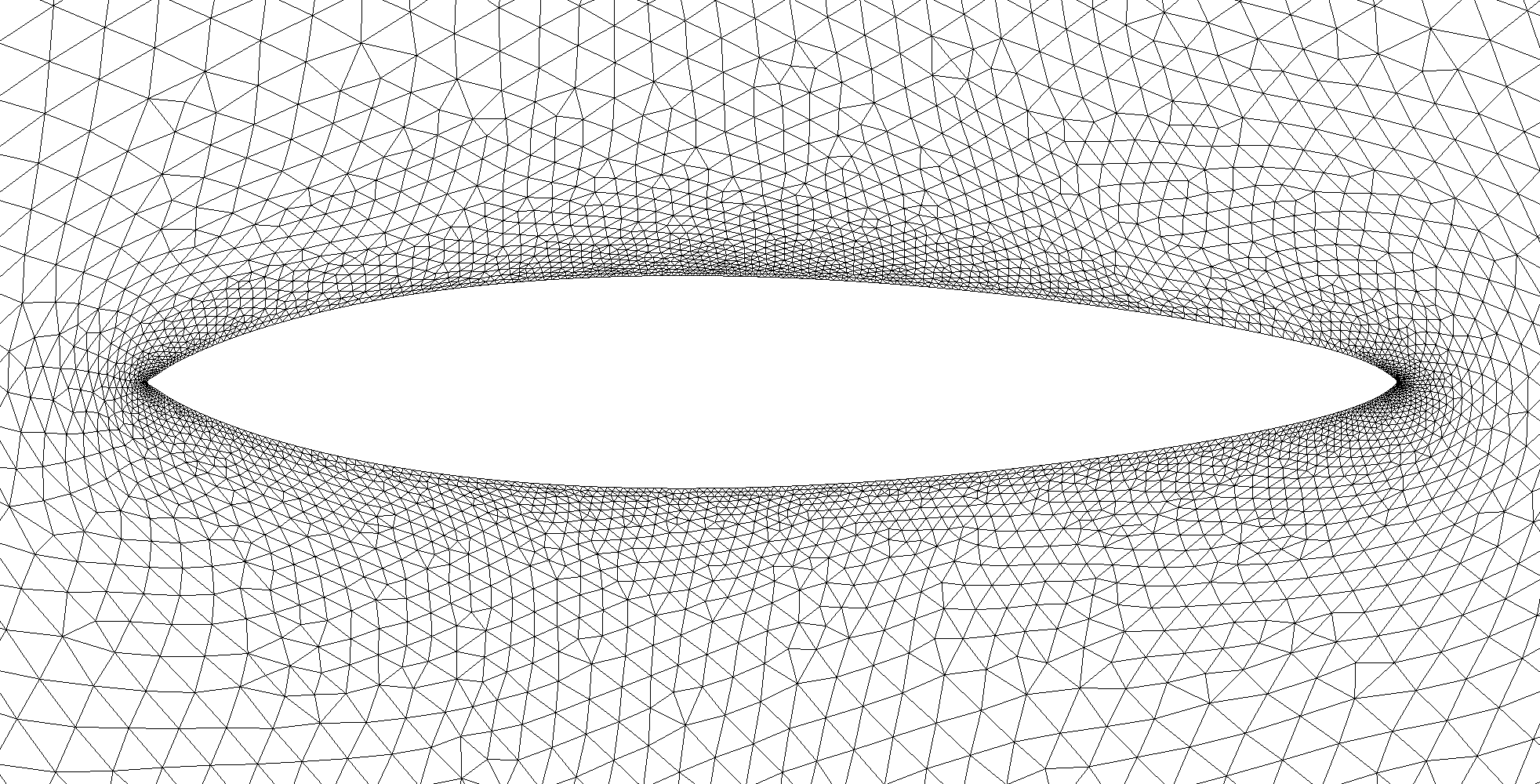}
		\subcaption{Deformed grid in $\F(\Omega)$}
	\end{subfigure}
	\caption{Mesh quality preserving shape optimization experiment with large deformation from reference to optimal configuration.}
	\label{fig::spline-mesh}
\end{figure}
Our first numerical study demonstrates the nonlinear extension equation for large deformations of a non-convex shape.
The reference domain $\Omega$ is chosen such that the specimen is described by a B-spline curve $\Gobs$ given in terms of 6 control nodes.
The situation is depicted in \cref{fig::spline-velocity,fig::spline-mesh}.

The relevance of this test case is to investigate the performance of the proposed approach for reference domains where the normal vector field $n$ does not homogeneously point in all directions as for a circular shape.
The influence of the normal vector is significant since it initially links the scalar-valued control $\control$ to a vector-valued quantity as can be seen in \cref{eq::nonlinear_extension_strong}.
\Cref{subsec::injectivity} is devoted to an experiment where several directions are underrepresented in the discretization of the normal vector $n$ due to the shape of $\Oobs$.

In \cref{fig::spline-velocity} the magnitude of velocity $\Vert v \Vert_2$, computed in the undeformed state $\Omega$, i.e.\ when $\control=0$, is depicted.
The right-hand side shows the velocity according to deformation $\F = \id + w$ in terms of the optimal control $\control$ after solving the optimality system given by \crefrange{eq::opt_sys_w}{eq::opt_sys_adj_bc}.
Furthermore, the optimal mapping $\F$ can be seen in \cref{fig::spline-mesh} in the displacement and deformation of discretization elements.
The relocation of triangles shows the effect of the nonlinear advection in the extension operator.
A deeper look on this effect and the resulting mesh quality is provided in \cref{subsec::quality-results}.

In this experiment the viscosity of the fluid is chosen as $\nu = \num{0.01}$.
The holdall domain $\holdall$ is as described above.
It consists of \num{382} surface segments on $\Gobs$ and \num{10106} triangles in $\Omega$.
Barycenter and volume of $\Oobs$ in the reference configuration are given by $\bc(\Oobs) = (\num{0.0307784},\num{-0.035759})^\top $ and $ \vol(\Oobs) = \num{0.809041}$.
Thus, an optimal shape also undergoes a small translation since $\bc(\F(\Oobs)) = 0$ is required.

The essential settings in terms of shape optimization are the parameter $\edet$ and $\eext$.
Here we choose $\eext = \num{3.0}$, which leads to the condition $\det(D\F) \geq \edet$ with $\edet = \num{5e-2}$ being inactive.
We can explain the homogeneously and smoothly deformed mesh due to this fact. 
In contrast, \cref{subsec::factor-influence} shows examples where $\det(D\F) \geq \edet$ is active close to the tip of the optimal shape and how the displacement $w$ and thereby the mesh quality in $\F(\Omega)$ is affected.

\subsection{Influence of factors $\edet$ and $\eext$}
\label{subsec::factor-influence}
\begin{figure}[h!]
	\begin{center}
		\includegraphics[width=0.7\textwidth]{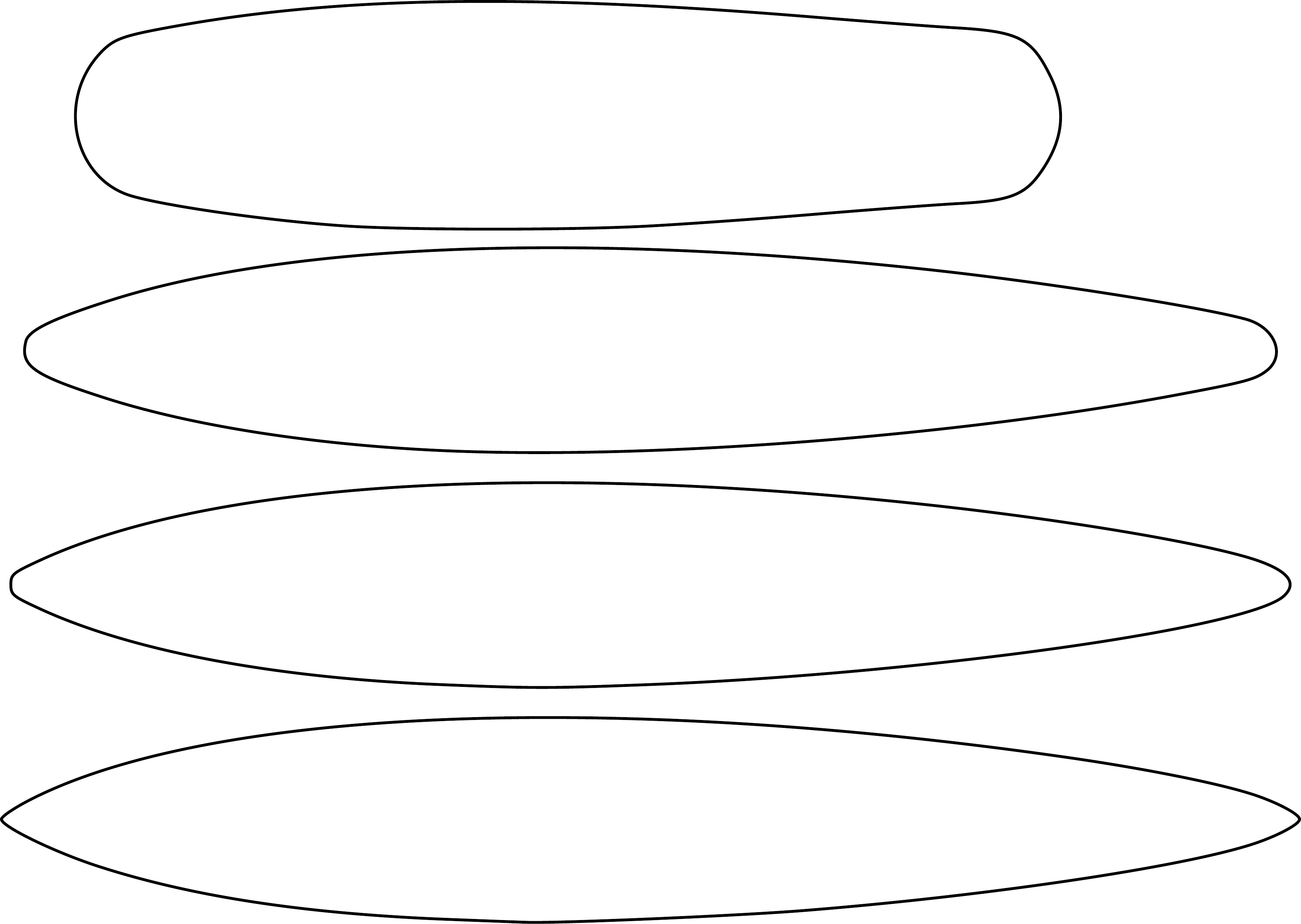}
	\end{center}
	\caption{Optimal shapes $\F(\Gobs)$ in terms of different $\edet \in \lbrace 0.5, 0.25, 0.2, 0.1 \rbrace$. The value $\edet$ decreases from top to bottom.}
	\label{fig::det-condition}
\end{figure}
\begin{figure}[h!]
	\centering
	\begin{subfigure}[c]{0.36\textwidth}
		\includegraphics[width=\textwidth]{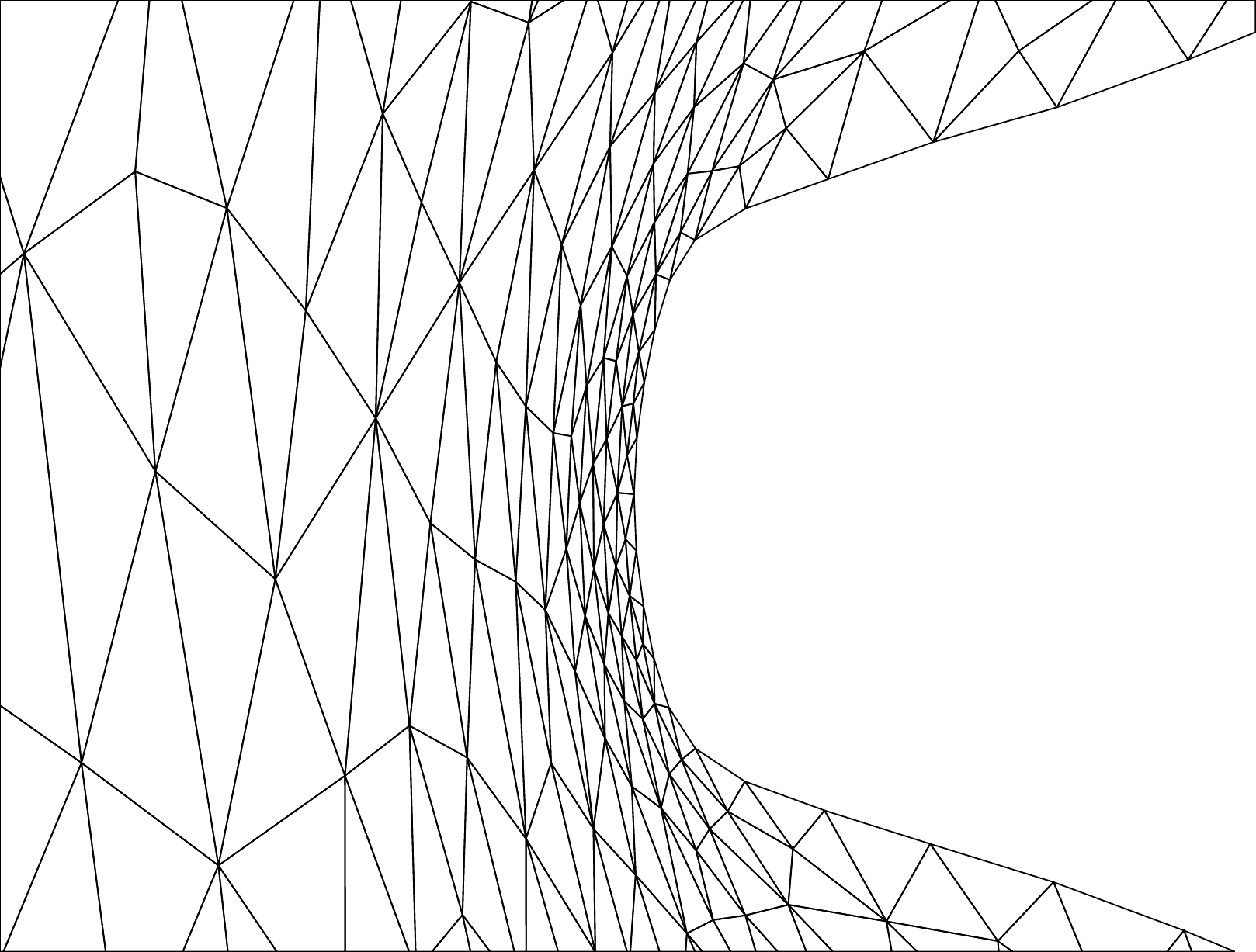}
		\subcaption{$\eext = 0.0$}
	\end{subfigure}
	\begin{subfigure}[c]{0.36\textwidth}
		\includegraphics[width=\textwidth]{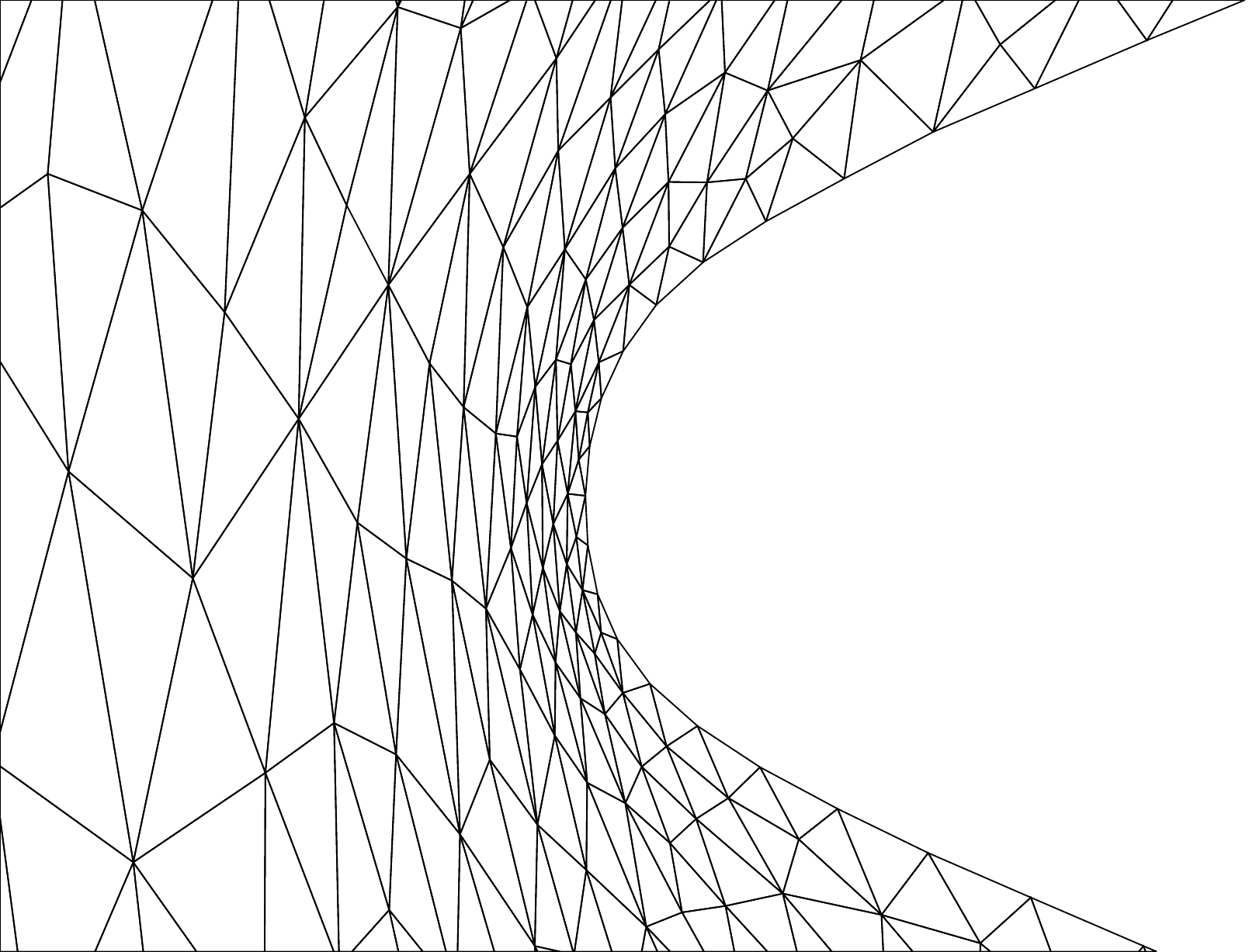}
		\subcaption{$\eext = 0.25$}
	\end{subfigure}
	\begin{subfigure}[c]{0.36\textwidth}
		\includegraphics[width=\textwidth]{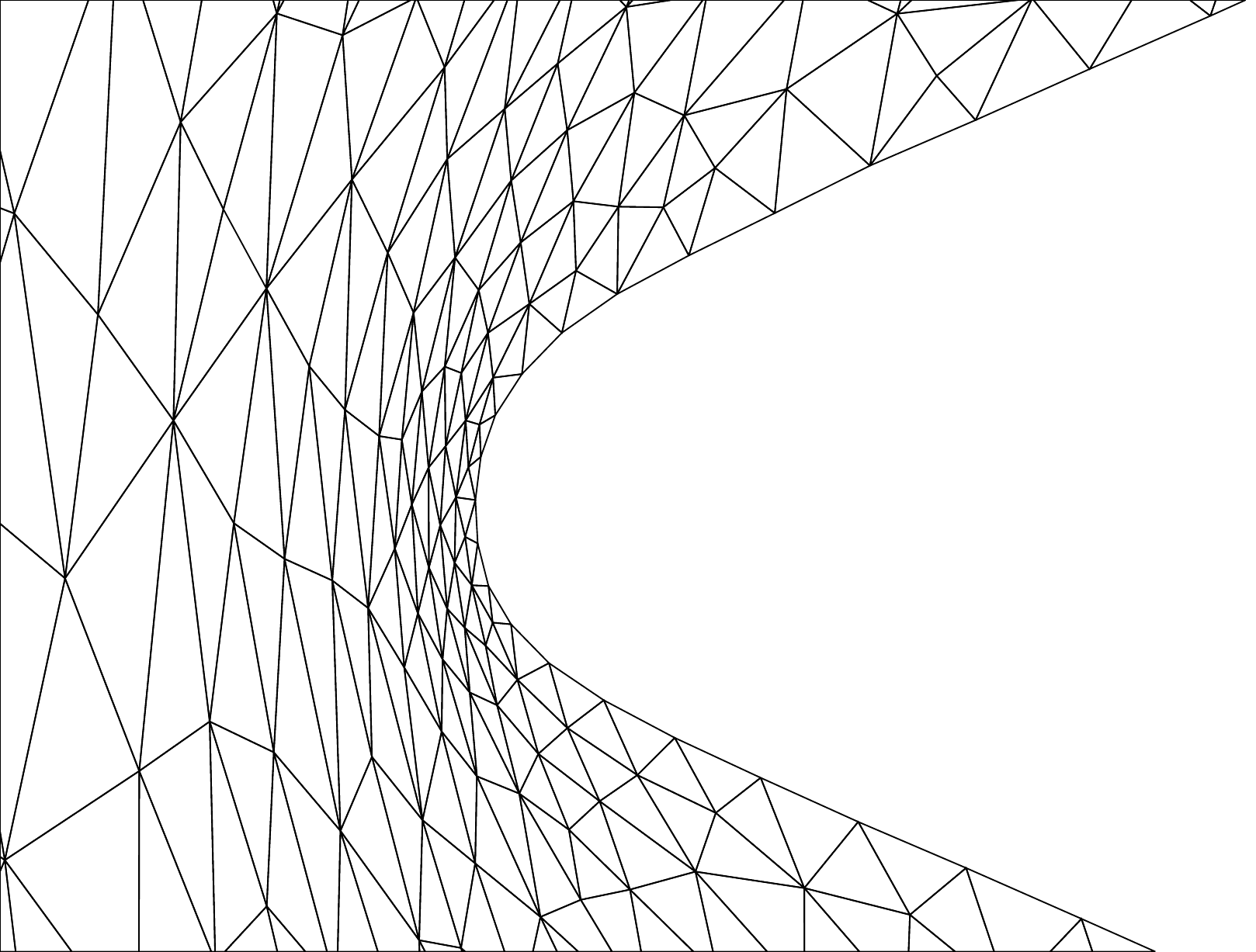}
		\subcaption{$\eext = 0.5$}
	\end{subfigure}
	\begin{subfigure}[c]{0.36\textwidth}
		\includegraphics[width=\textwidth]{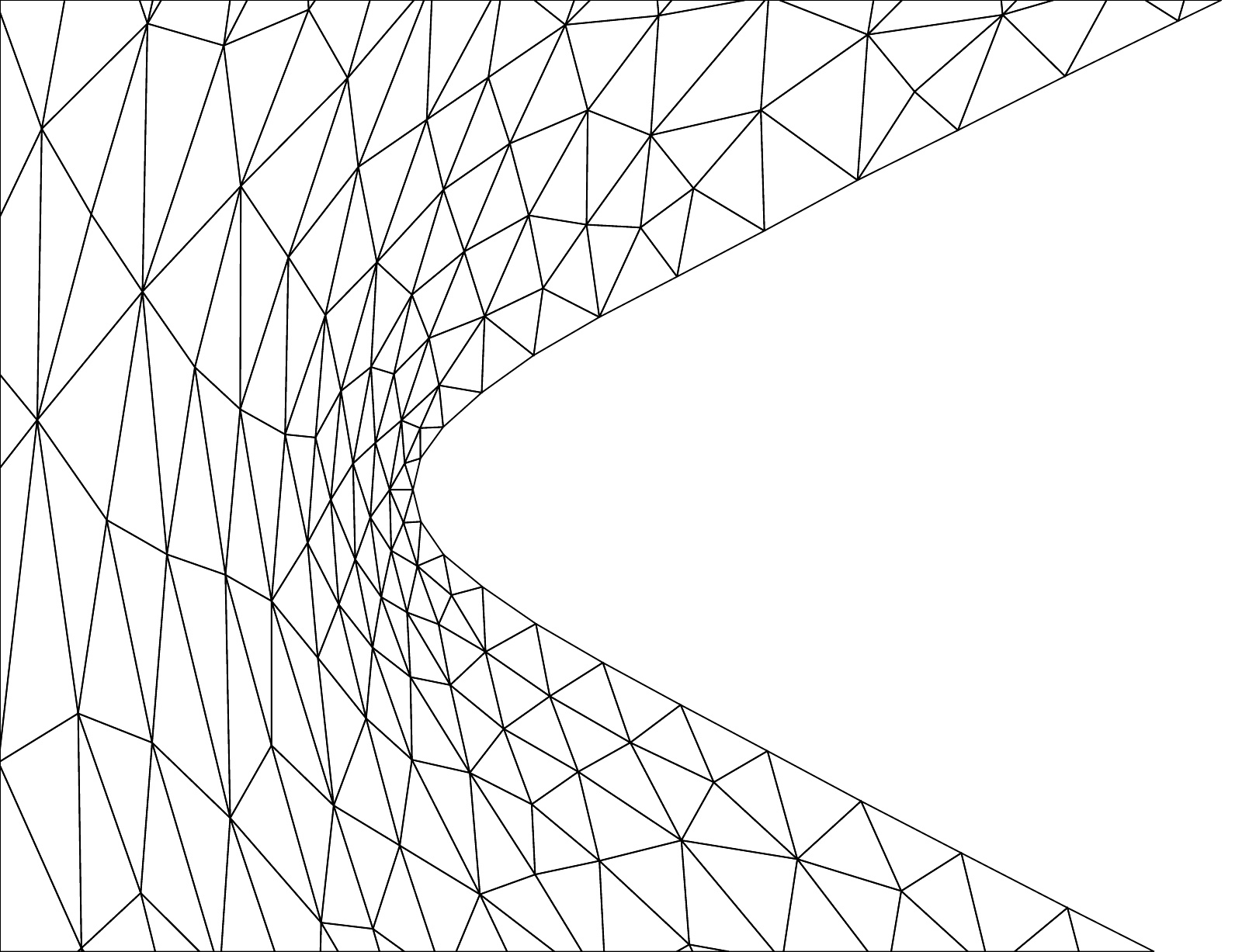}
		\subcaption{$\eext = 1.0$}
	\end{subfigure}
	\begin{subfigure}[c]{0.36\textwidth}
		\includegraphics[width=\textwidth]{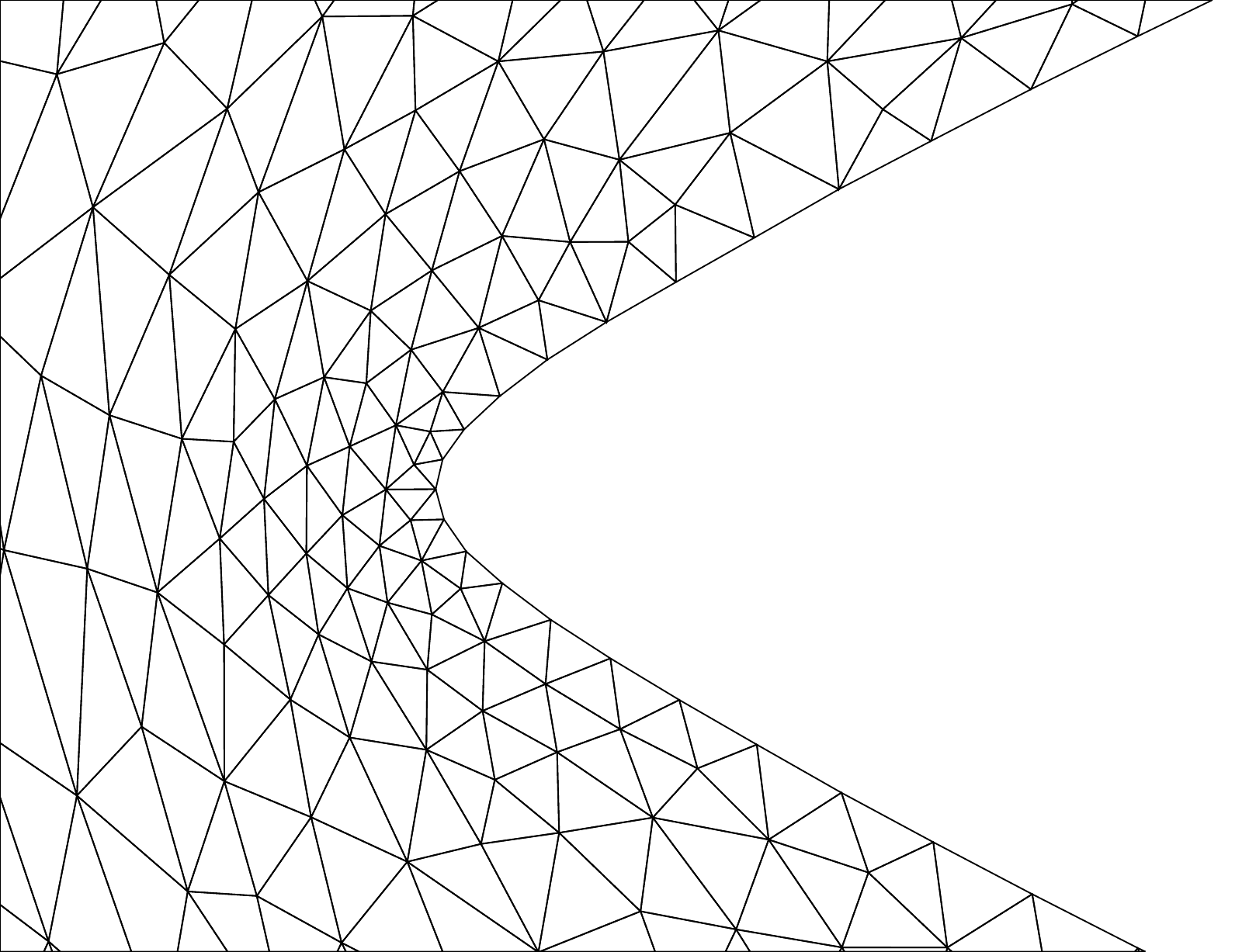}
		\subcaption{$\eext = 2.0$}
	\end{subfigure}
	\begin{subfigure}[c]{0.36\textwidth}
		\includegraphics[width=\textwidth]{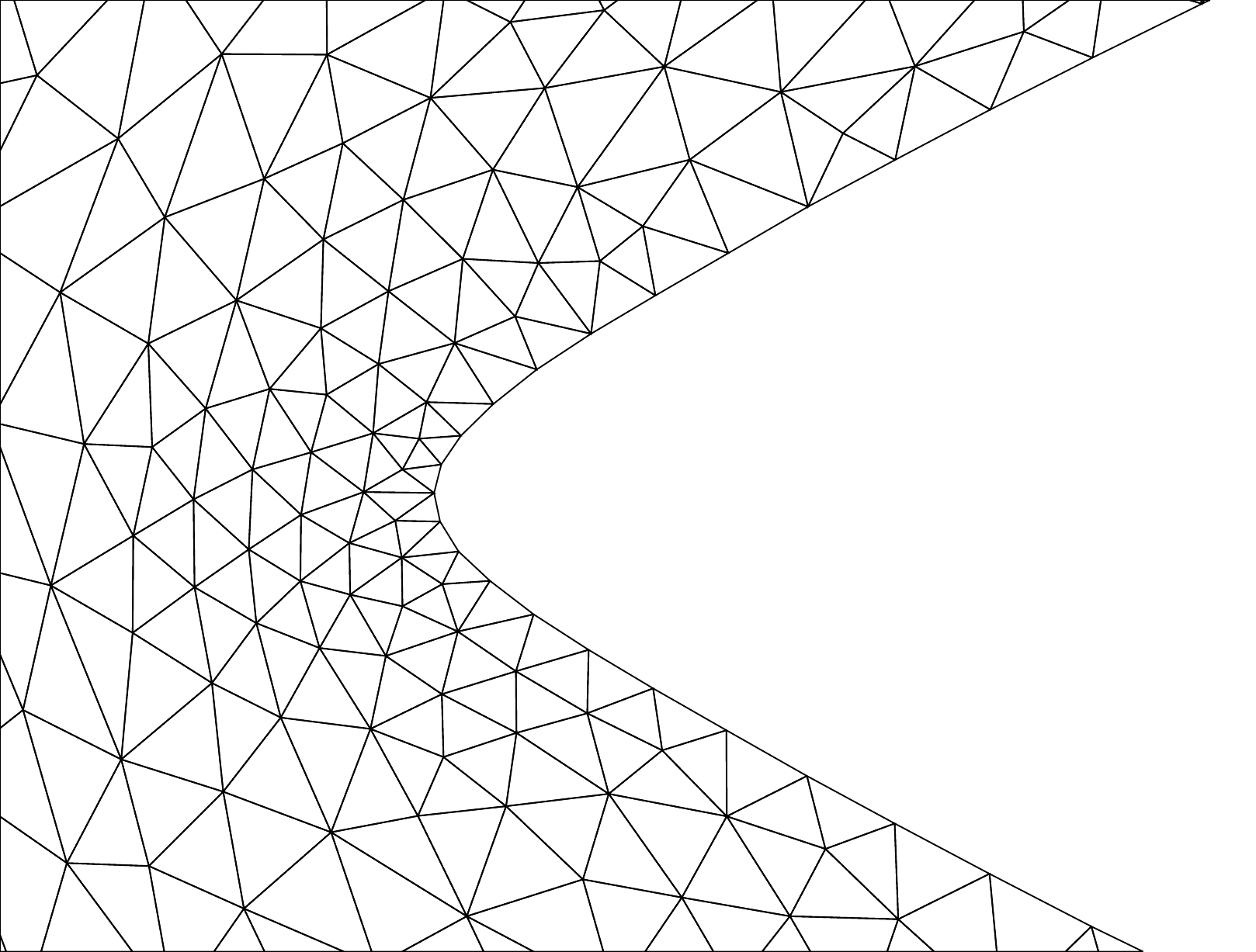}
		\subcaption{$\eext = 3.0$}
	\end{subfigure}
	\caption{Experiments with $\edet = 0.05$ constant and variable $\eext$ factor.}
	\label{fig::extension-factor}
\end{figure}
In this section we visualize the influence of the choice of $\edet$ and $\eext$ on the optimization. This demonstrates how the set of admissible shapes $\mathcal{F}_\text{adm}$ is determined thereby.
The underlying experiment is a flow in $\holdall$ over a circular specimen as described at the beginning of \cref{sec::numerics}.
The viscosity is again chosen to be $\nu = \num{0.01}$.
The domain is discretized with \num{244} segments on $\Gobs$ and \num{12640} triangles in $\Omega$.

First, we observe the influence of $\edet$ on the set of admissible shapes $\mathcal{F}_\text{adm}$.
\Cref{fig::det-condition} visualizes how the condition $\det(D\F) \geq \edet$ acts on the optimal shape $\F(\Gobs)$.
Here we choose $\edet \in \lbrace 0.5, 0.25, 0.2, 0.1 \rbrace$ beginning with the largest and then decreasing values.
This condition can be interpreted such that the allowed, local change of volume in $\Omega$ is relaxed from top to bottom in \cref{fig::det-condition}.
In this experiment it turns out that in the last computation with $\edet = \num{0.1}$ the condition is inactive.
Here $\eext= \num{3.0}$ is fixed in all computations.

The next experiment follows the same setup with the only difference that $\edet = \num{5e-2}$ is now fixed and $\eext \in \lbrace \num{0.0}, \num{0.25}, \num{0.5}, \num{1.0}, \num{2.0}, \num{3.0}\rbrace$ takes increasing values.
The mesh deformations, resulting from the optimal solution $\F = \id + w$, are visualized in \cref{fig::extension-factor}.
Each of the subfigures shows a clip of size $\num{0.1}\times\num{0.08}$ around the tip of the optimal shape.
Besides the significantly improved mesh qualities and more adequate set $\mathcal{F}_\text{adm}$ we also observe that the Newton solver benefits from the appropriate choice of $\edet$ and $\eext$.
As soon as the condition $\det(D\F) \geq \edet$ becomes active, the optimality system \crefrange{eq::opt_sys_w}{eq::opt_sys_adj_bc} is not differentiable any further and the solver switches to semismooth Newton's method.
This effect is already documented in \cite{haubner2020continuous} for the case of Stokes flows and linear extension operator $S$.

\subsection{Extending the local-only injectivity}
\label{subsec::injectivity}
\begin{figure}[h!]
	\begin{subfigure}[t]{0.32\textwidth}
		\includegraphics[width=\textwidth]{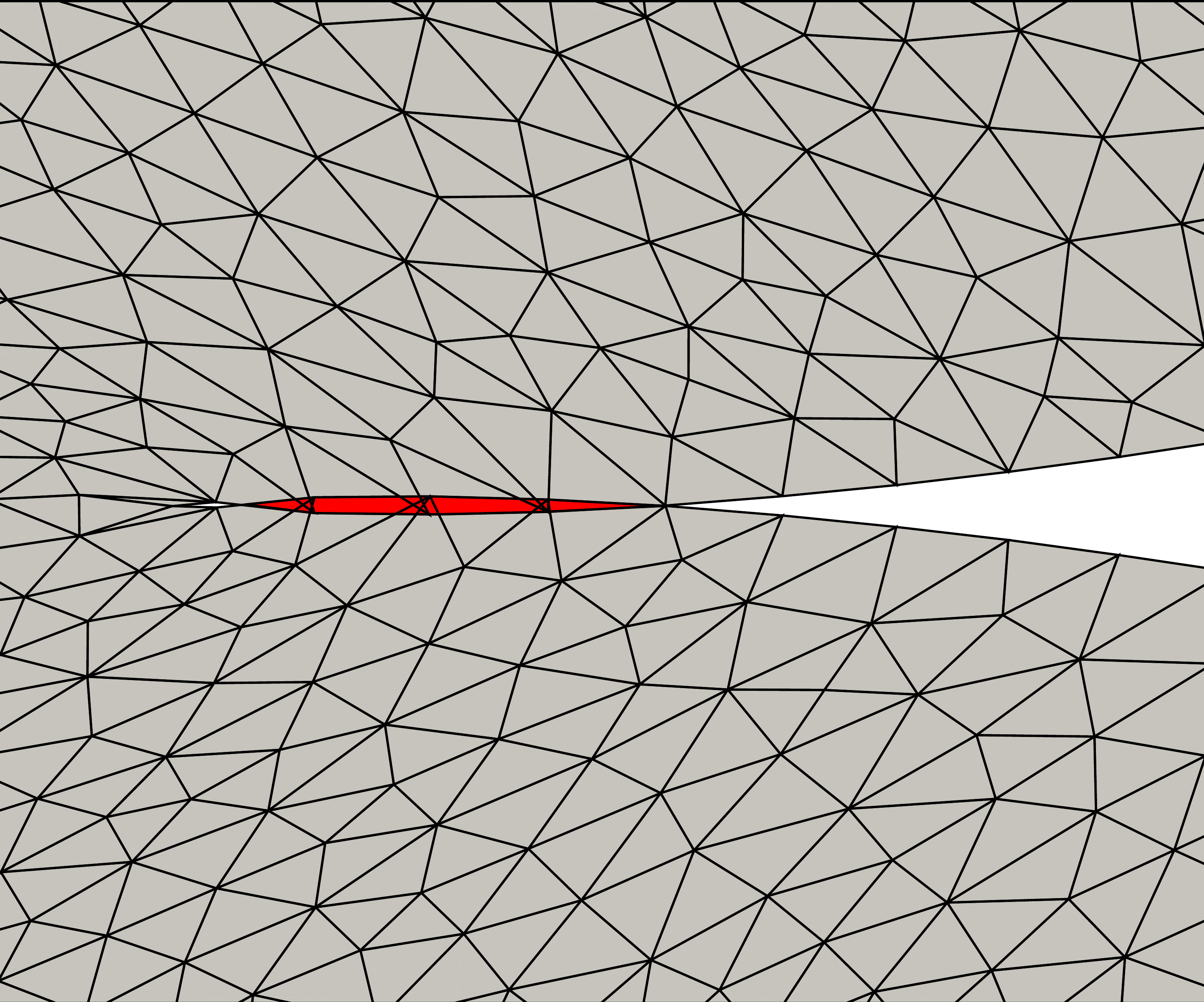}
		\subcaption{Overlap of cells (red) in $\F(\Omega)$}
		\label{fig::overlappings}
	\end{subfigure}
	\begin{subfigure}[t]{0.32\textwidth}
		\includegraphics[width=\textwidth]{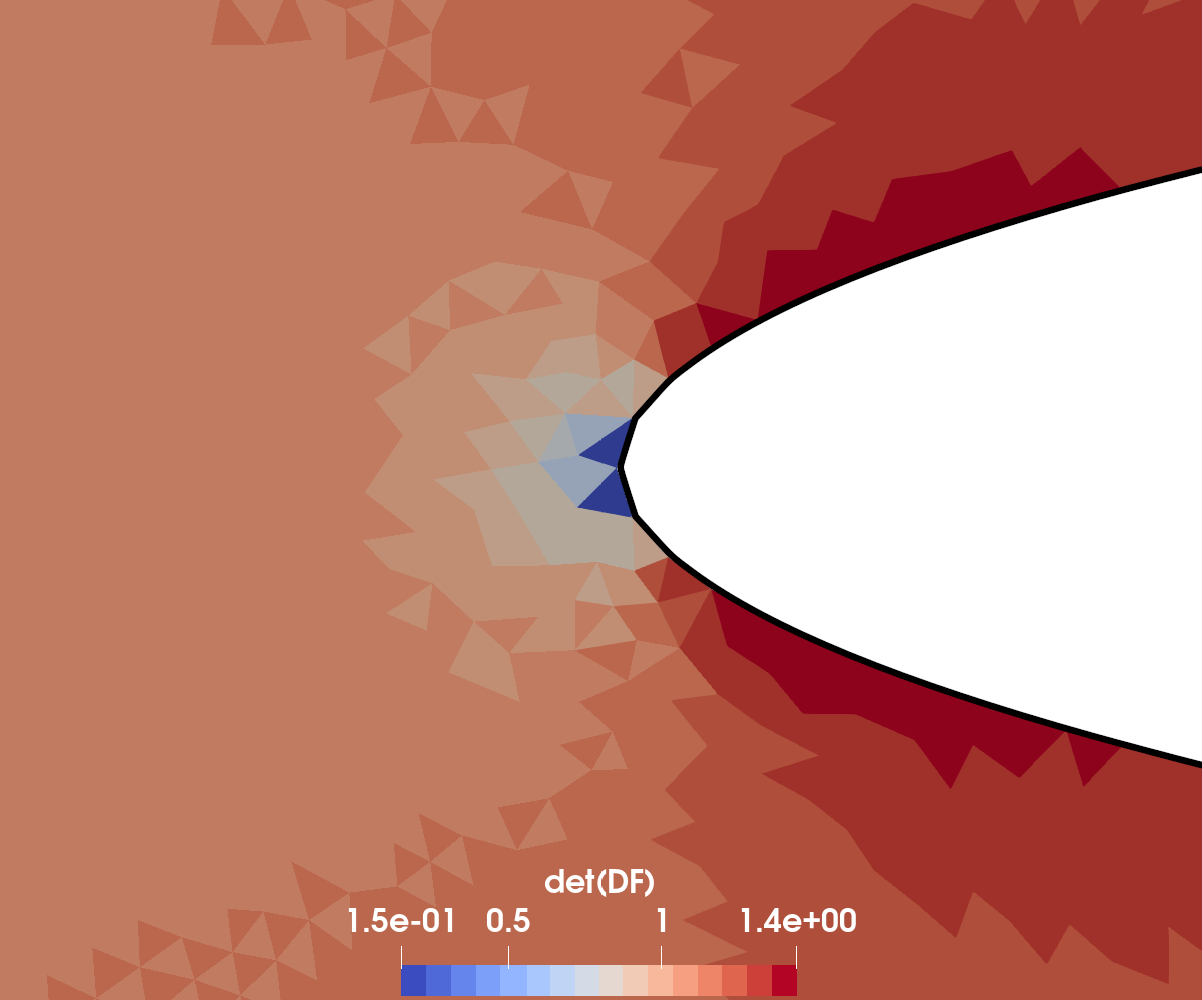}
		\subcaption{$\det(D\F)$ for hollow $\Oobs$}
		\label{fig::det-ellipse-hollow}
	\end{subfigure}
	\begin{subfigure}[t]{0.32\textwidth}
		\includegraphics[width=\textwidth]{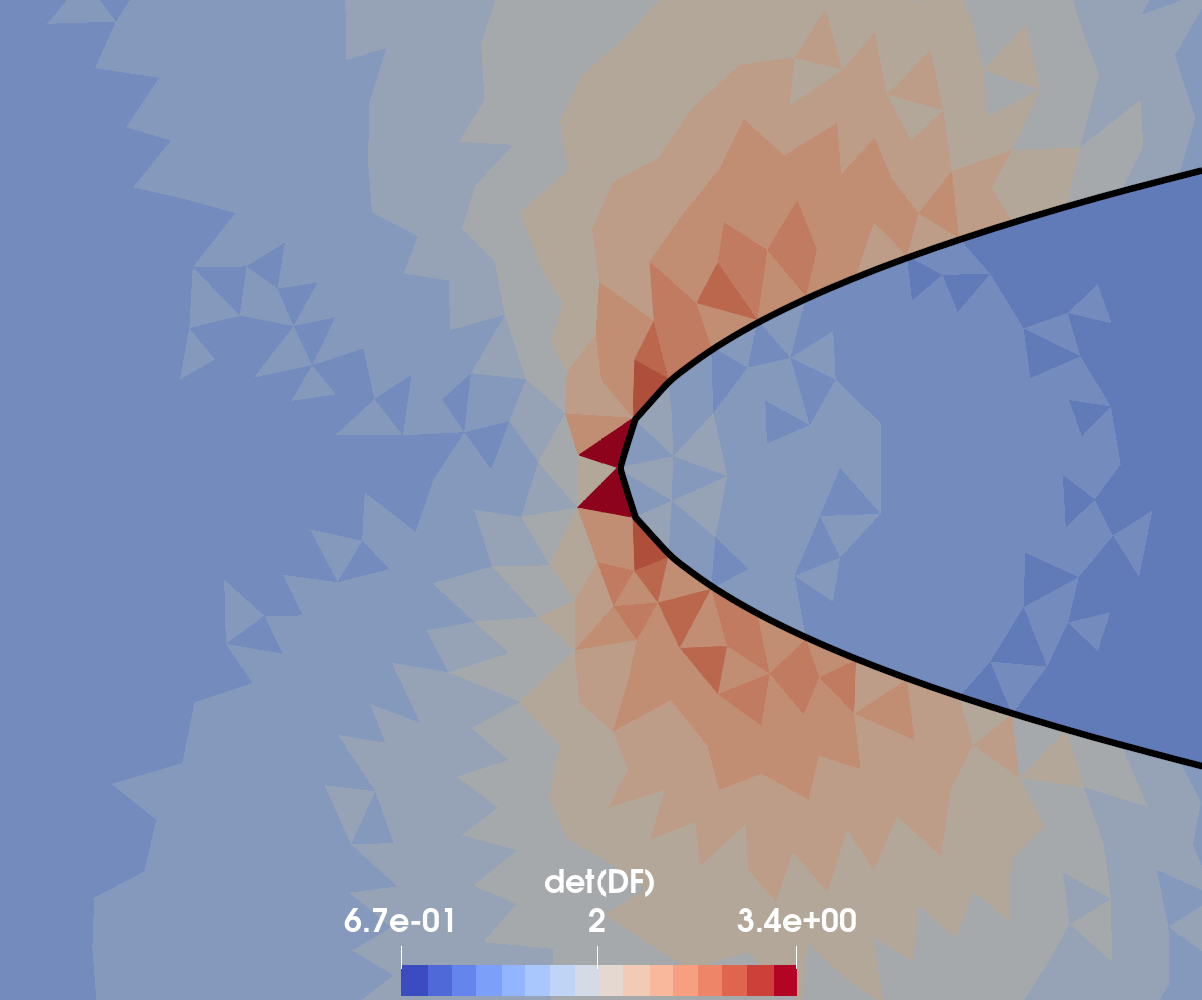}
		\subcaption{$\det(D\F)$ for discretized $\Oobs$}
		\label{fig::det-ellipse-filled}
	\end{subfigure}
	\caption{Closeup visualization of the discretization around the tip of the ellipse experiment.}
	\label{fig::ellipse-nose}
\end{figure}
\begin{figure}[h!]
	\begin{subfigure}{\textwidth}
		\centering
		\includegraphics[width=0.7\textwidth]{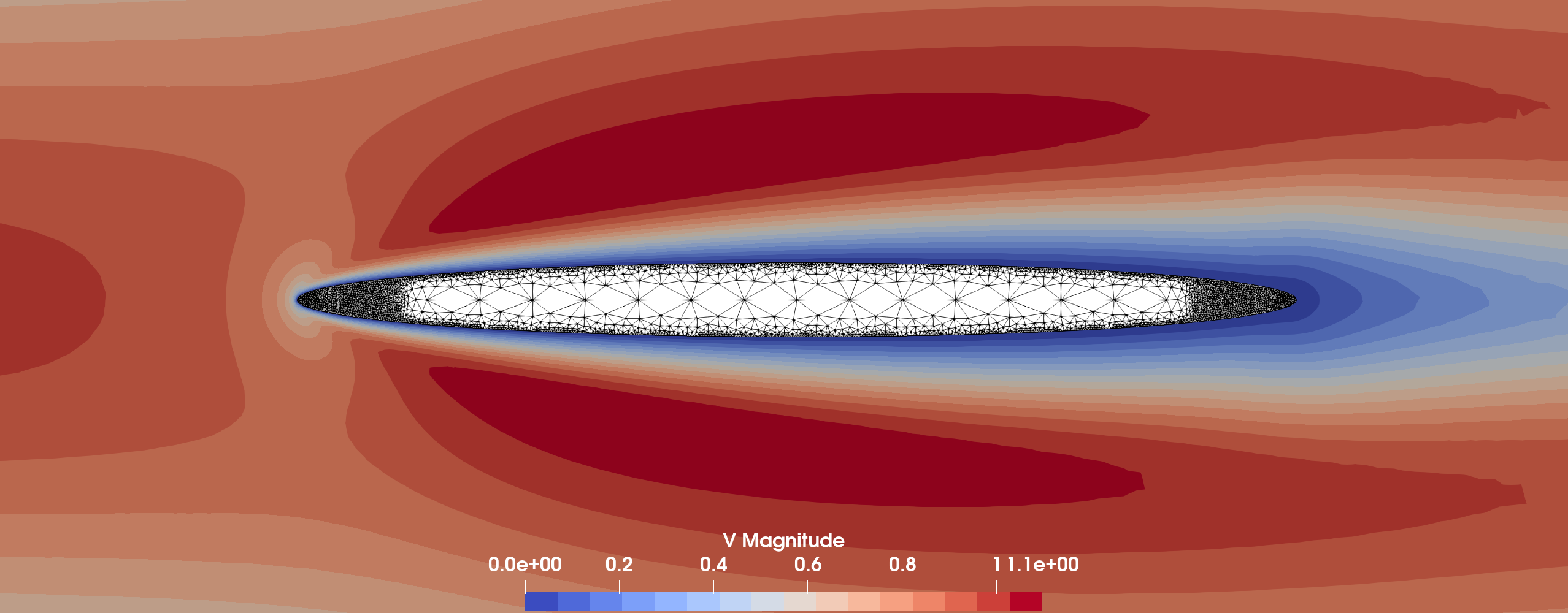}
		\subcaption{Reference domain}
		\label{fig::ellipse-reference}
	\end{subfigure}\\[0.2cm]
	\begin{subfigure}{\textwidth}
		\centering
		\includegraphics[width=0.7\textwidth]{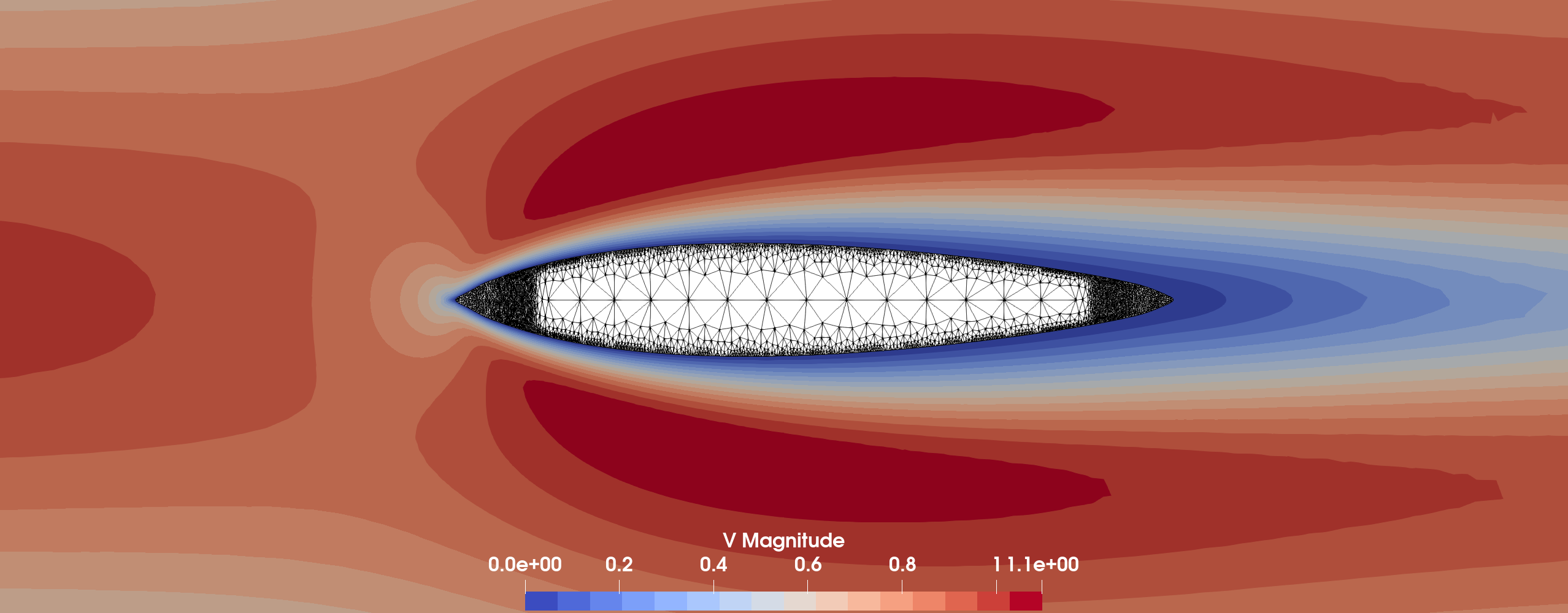}
		\subcaption{Grid deformed according to optimal displacement $w = S(c)$}
		\label{fig::ellipse-filled}
	\end{subfigure}
	\caption{Flow over ellipsoidal reference shape and optimal solution with visualization of auxiliary grid in $\Oobs$. Color denotes norm of velocity field $\Vert v \Vert_2$.}
	\label{fig::ellipse-experiment}
\end{figure}
This section focuses on an effect that is likely to appear for non-spherical reference domains.
In particular, for the aerodynamic experiments considered in this article it might happen that the upper surface of the obstacle overlaps the lower one.
Especially for large deformations from reference to optimal shape and for shapes that are streched parallel to the flow axis, we encounter effects as depicted in \cref{fig::overlappings} for the experiment shown in \cref{fig::ellipse-experiment}.
In other words, $\Omega$ $\mapsto \F(\Omega)$ is not globally injective in this situation.
This is due to the fact that the condition $\det(D\F)$ ensures injectivity of $\F$ only locally but not globally.

In the following we propose a modification of the extension operator $S$ in order to extend the injectivity.
Recall that in the setting followed up to here the obstacle domain $\Oobs$ is treated as void and there is no discritization within.
We now consider the operator $S$ on the entire holdall domain $\holdall = \Omega \cup \Oobs$ in contrast to the state equation that remains in $\Omega$.
Moreover, the condition $\det(D\F) \geq \edet > 0$ is now required on $\holdall$.
Thus, the displacement field is defined by $w \in H_0^1(\holdall, \R^d)$.
Moreover, we reformulate the weak formulation of the extension operator $S$ given in \cref{eq::nonlinear_ext} to
\begin{equation}
\begin{aligned}
\int_{\holdall} (Dw + Dw^\top):D\test{w} + \eext(Dw\, w)\cdot \test{w} \, dx = \int_\Gobs b \test{w} \, ds
\end{aligned}
\end{equation}
for all $\test{w} \in H_0^1(\holdall, \R^d)$ and appropriate $\eext \geq 0$.
Simultaneously, the penalty term, which enforces the local injectivity, in \cref{eq::Lagrangian} changes to 
\begin{equation}
\frac{\beta}{2} \int_{\holdall} (( \edet - \det(D\F) )_+ )^2 \, dx.
\end{equation}

For this experiment we choose $\nu = \num{0.01}$ as in the previous sections.
The holdall $\holdall$ has the same outer dimension and the specimen $\Gobs$ is an ellipse with semimajor-axis $r_1=\num{2.7}$, semiminor-axis $r_2=\num{0.2}$ and barycenter $\bc(\Oobs) = (0, 0)^\top$.
Its surface is subdivided in \num{884} segments.
Further, $\holdall$ is discretized by \num{36360} triangles, \num{29930} in $\Omega$ and \num{6430} in $\Oobs$.

\Cref{fig::ellipse-nose} visualizes the effect of the mapping $\F$ on the discretization grid.
In \cref{fig::overlappings} the optimal solution for $\alpha=\num{1e-2}$ is shown.
Note that in this particular case we stop the optimization for a larger value, since this already leads to singularities.
\Cref{fig::det-ellipse-hollow} depicts $\det(D\F)$, which is again bound away from zero by $\edet=\num{5e-2}$.
It can be seen that, although this condition is inactive, the non-injective mapping can not be prevented.

The same experiment is then conducted with the changes proposed in the beginning of this section, which leads to the values of $\det(D\F)$ shown in \cref{fig::det-ellipse-filled}.
Now $\Oobs$ is discretized and $S$ also acts on the interior of the specimen.
Here the optimization is performed with the setting $\ainit=\num{1e-4}$, $\adec=\num{5e-1}$ and $\atarget=\num{1e-10}$.
The resulting optimal solution is visualized in \cref{fig::ellipse-experiment} where \cref{fig::ellipse-reference} shows the reference domain and the velocity field computed for this configuration.
\Cref{fig::ellipse-filled} depicts the domain $\F(\holdall)$ and the velocity field computed on $\F(\Omega)$.
Note that the relatively fine grid is chosen at the front and the back of the shape due to the large curvature of $\Gobs$ in these regions.
This experiment turns out to be more challenging than, e.g., a spherical reference shape since on coarse grids the normal vector field in these areas tends to be underresolved.
From a computational point of view, it is attractive to have a coarse grid in $\Oobs$, as chosen in the center of the specimen, to reduce cost for the solution of the operator $S$.

\subsection{Three-dimensional results}
\label{subsec::3d-results}
In this section we perform a three-dimesional optimization experiment as a proof of concept.
As already observed in \cite{haubner2020continuous} for the Stokes experiment, more care has to be taken for the decrease-strategy of $\alpha$ in \cref{alg::direct}.
Especially the semismooth Newton solver shows to be challenging w.r.t.\ to convergence when the condition $\det(D\F) \geq \edet$ becomes active.

The experiment shown in \cref{fig::3d-result} is within the framework described at the beginning of \cref{sec::numerics}.
The flow tunnel $\Omega$ is discretized by \num{632093} tetrahedrons and the surface of the spherical specimen in the reference configuration $\Gobs$ consist of \num{8558} triangles.
Further, the viscosity is chosen to be $\nu=\num{0.01}$ and in \cref{alg::direct} we set $\ainit = \num{1e-4}$, $\adec = \num{5e-1}$ and $\atarget=\num{1e-6}$.
The results shown here are obtained with an extension factor of  $\eext=15$.
We visualize the impact of the optimization on the fluid by stream lines of the velocity field in \cref{fig::3d-result}.
This figure also shows the effect of the particular operator $S$ on the quality of the surface mesh when it undergoes the optimal deformation $\F$.
The combination of Laplace-Beltrami \cref{eq::nonlinear_ext_lb} and the nonlinear extension equation \cref{eq::nonlinear_ext_lb} leads to a homogeneous distribution of triangles on the surface $\Gobs$.
It can be observed that this is due to tangential components in $w\vert_\Gobs$.
This is a benefit of a vector-valued extension equation over approaches which utilize a static extension of the normal vector field in order to extend the boundary control to the surrounding volume.
Furthermore, \cref{fig::3d-mesh-nose} shows a zoom-in to the tip of the deformed domain $\F(\Omega)$.
Here we can see a crinkled clip in the $x_1 x_2$-plane with $x_3=0$, which shows the quality of the tetrahedrons.
\begin{figure}[h!]
	\centering
	\begin{subfigure}{0.7\textwidth}
		\includegraphics[width=\textwidth]{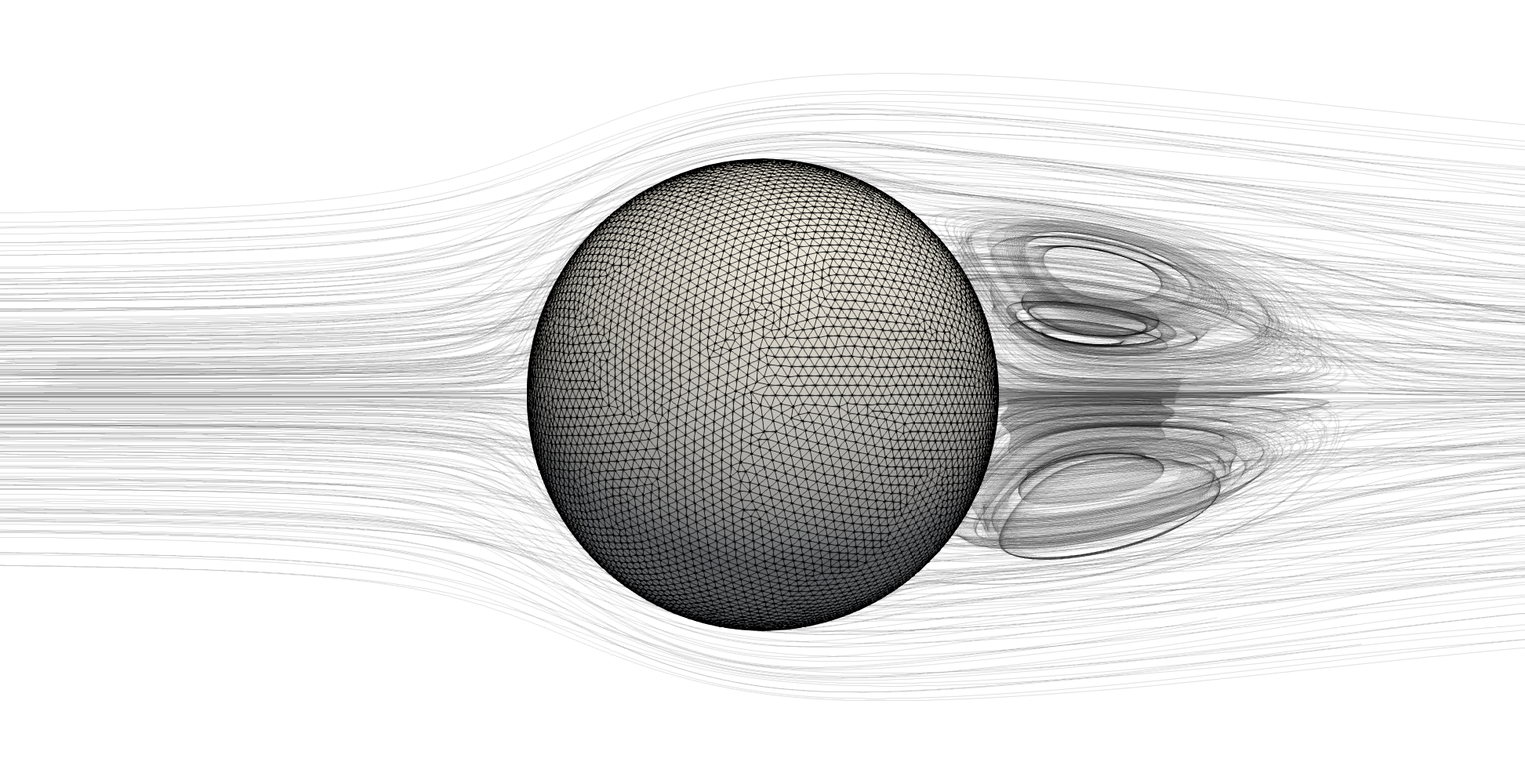}
		\subcaption{Surface grid of reference domain.}
	\end{subfigure}
	\begin{subfigure}{0.7\textwidth}
		\includegraphics[width=\textwidth]{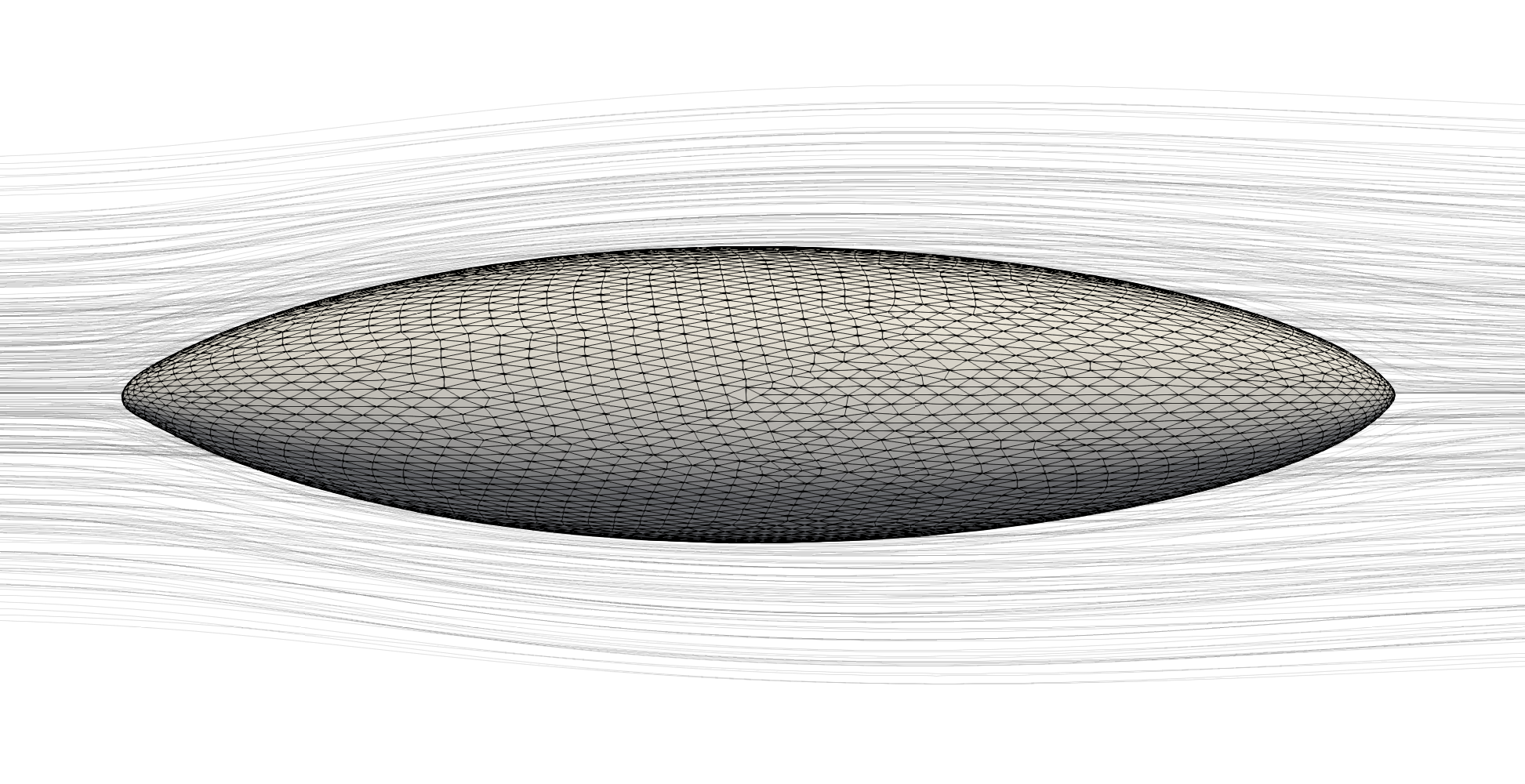}
		\subcaption{Optimal solution and deformed surface grid $\F(\Gobs)$.}
	\end{subfigure}
	\caption{Velocity stream lines computed on reference and optimal domain together with visualization of surface discretization.}
	\label{fig::3d-result}
\end{figure}

\begin{figure}[h!]
	\centering
	\includegraphics[width=0.4\textwidth]{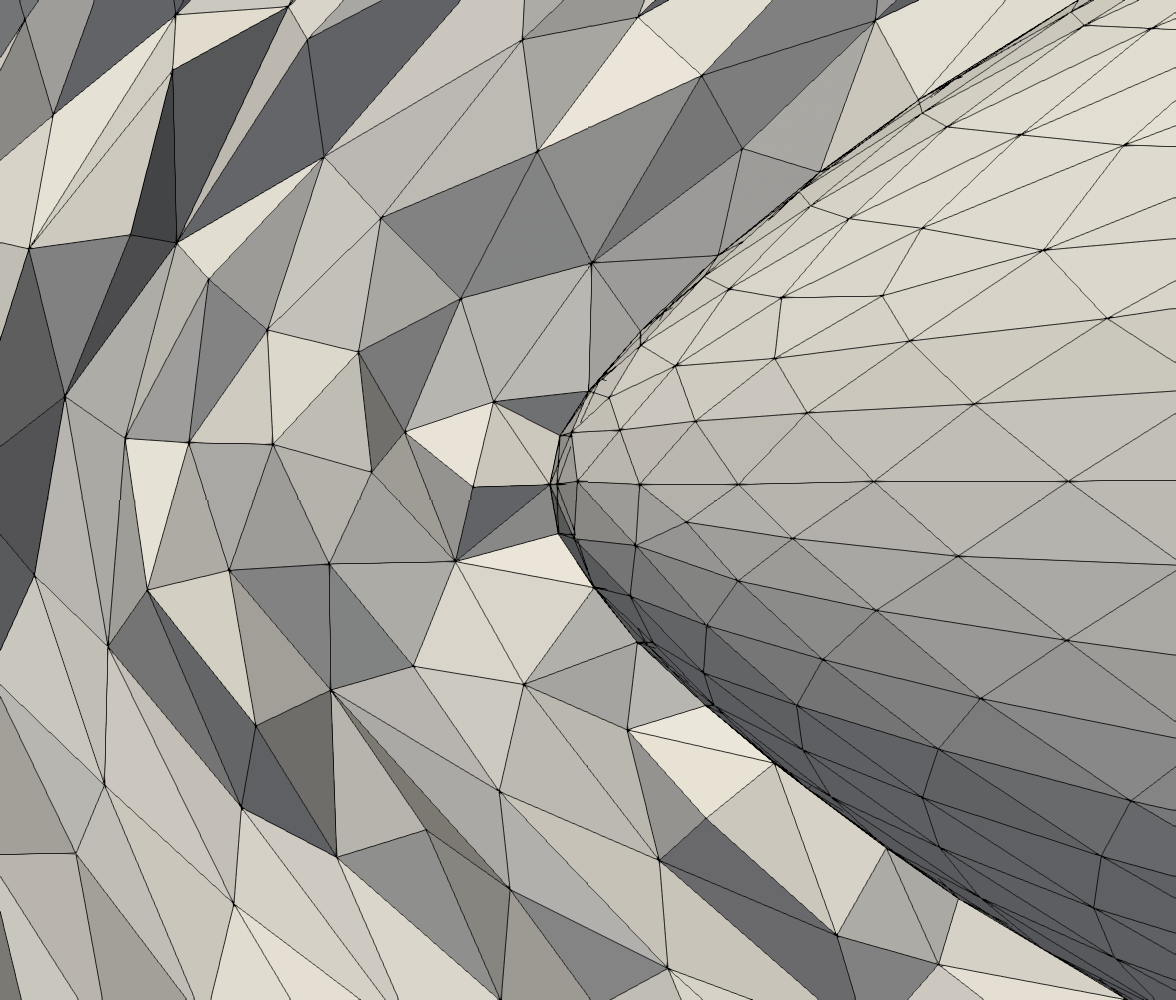}
	\caption{Crinkled clip of $x_1 x_2$-plane with $x_3=0$ showing deformed mesh $\F(\Omega)$ together with surface elements $\F(\Gobs)$.}
	\label{fig::3d-mesh-nose}
\end{figure}

\subsection{Quantification of the influence of $\eext$ on mesh quality}
\label{subsec::quality-results}
This section presents numerical experiments which investigate the influence of the nonlinear extension operator $S$ on the mesh quality in 2d and 3d (cf. \cref{fig::mesh_quality}).
Recall that the discretization mesh is not actually deformed within the optimization.
We though deform the reference domain $\Omega$ according to the optimal control $w = S(\control)$ and the corresponding deformation $\F = \id + w$.
The 2d experiment is conducted on the same computational domain as before with a circular specimen, \num{312} surface segments and \num{6168} triangles in $\Omega$.
The fluid viscosity is chosen to be $\nu = \num{0.01}$.
\Cref{fig::quality_2d} visualizes the influence of $\eext \in \lbrack0, 1 \rbrack$ on the mesh quality of $\F(\Omega)$.
It is measured by the ratio of radii of largest inscribed and smallest circumscribed circle, where the plot shows the value of the worst triangle.

This experiment quantifies the effect which is already visualized in \cref{fig::extension-factor}.
For a shape optimization with large deformations from reference to optimal configuration, i.e.\ $\Vert w \Vert_{L^2(\Gobs)}$ is relatively large, a pure linear extension operator $S$ does not reliably lead to satisfying mesh qualities.
Moreover, it can be seen that in this particular experiment there is a saturation effect of the nonlinearity in $S$ starting at approximately $\eext \approx \num{1.5}$.
\Cref{fig::quality_3d} shows the results of a similar experiment in 3d.
Here a mesh is chosen with \num{6040} surface triangles on $\Oobs$ and \num{147 385} tetrahedrons in $\Omega$.
Note that we decrease the viscosity to $\nu = 0.1$ in this experiment in order to be able to obtain results for $\eext < 3.0$.
In 3d quality is measured by the radius ratio of smallest circumscribed sphere to the largest inscribed one.
Again the worst element is visualized.
Also note that the y-axis is in log-scale.
In this setting it turns out, that the effect of compressed cells near the tip and back of the shape, which is stretching due to a decrease in $\alpha$, is stronger than in 2d.
We explain the solver failure due to the semismoothness in the optimality system, which becomes active in a significant number of finite elements in this situation.
However, it can be observed that, starting with approximately $\eext \approx 8$, a saturation is possible, where the mesh quality of $\F(\Omega)$ remains adequate for further numerical computations.
\pgfplotsset{width=1.0\textwidth,compat=1.9}
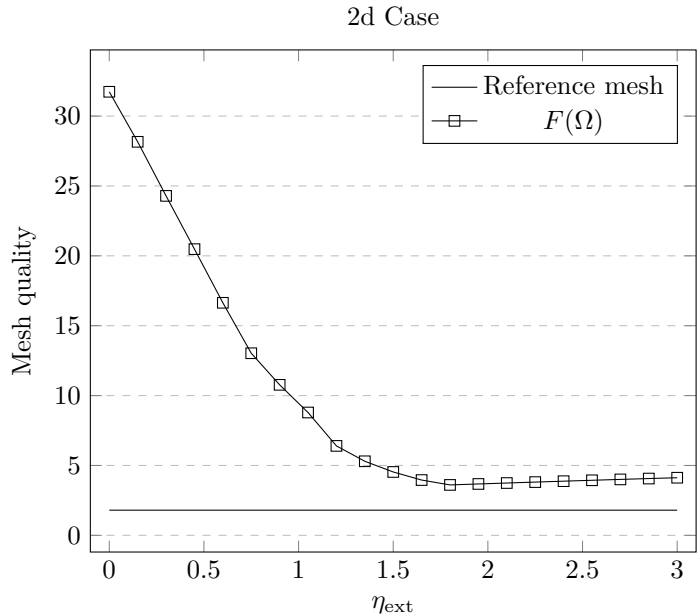
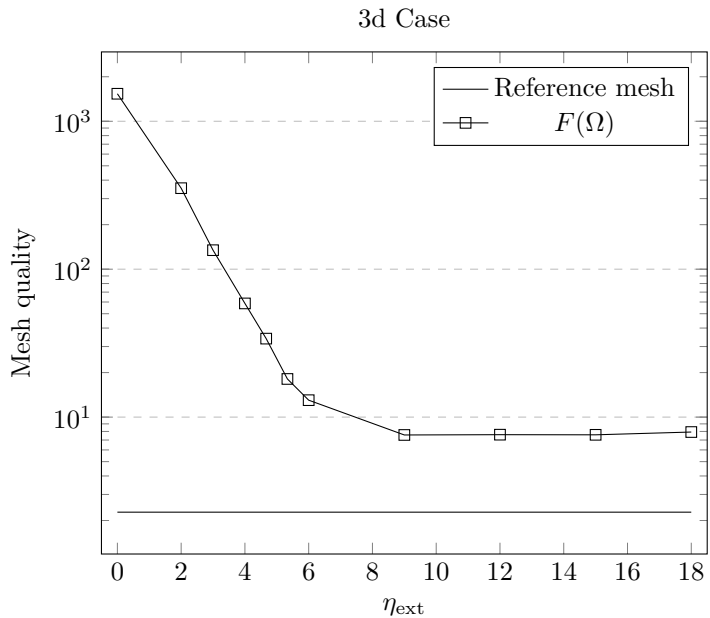
\begin{figure}[h!]
	\centering
	\begin{minipage}{0.6\textwidth}
		\centering
		\begin{tikzpicture}
		\begin{axis}[
		title={2d Case},
		ylabel={Mesh quality},
		xlabel={$\eext$},
		xmin=-0.1, xmax=3.1,
		legend pos=north east,
		ymajorgrids=true,
		grid style=dashed,
		ymode=linear,
		]
		\addplot[mark=none, black] table{
			0 1.8
			3 1.8
		};
		\addlegendentry{Reference mesh}
		\addplot[
		color=black,
		mark=square,
		]
		table {
			0	31.74
			0.15	28.16
			0.3	24.29
			0.45	20.48
			0.6	16.64
			0.75	13.03
			0.9	10.77
			1.05	8.79
			1.2	6.4
			1.35	5.3
			1.5	4.5345
			1.65	3.9565
			1.8	3.60772
			1.95	3.676
			2.1	3.74356
			2.25	3.81015
			2.4	3.87561
			2.55	3.93977
			2.7	4.00244
			2.85	4.0635
			3	4.12284
		};
		\addlegendentry{$\F(\Omega)$}
		\end{axis}
		\end{tikzpicture}
		\subcaption{Quality measured by ratio of circumscribed/inscribed cricle radii.}
		\label{fig::quality_2d}
	\end{minipage}
	\vspace{0.5cm}
	\begin{minipage}{0.6\textwidth}
		\centering
		\begin{tikzpicture}
		\begin{axis}[
		title={3d Case},
		ylabel={Mesh quality},
		xlabel={$\eext$},
		xmin=-0.5, xmax=18.5,
		legend pos=north east,
		ymajorgrids=true,
		grid style=dashed,
		ymode=log,
		]
		\addplot[mark=none, black] table{
			0 2.277
			18 2.277
		};
		\addlegendentry{Reference mesh}
		\addplot[
		color=black,
		mark=square,
		]
		table {
			0		1534.93
			2		352.988
			3		134.44
			4		58.6956
			4.667	33.9368
			5.333	18.1106
			6		13.0168
			9		7.57392
			12		7.6078
			15		7.59228
			18		7.92797
		};
		\addlegendentry{$\F(\Omega)$}
		\end{axis}
		\end{tikzpicture}
		\subcaption{Quality measured by ratio of circumscribed/inscribed sphere radii.}
		\label{fig::quality_3d}
	\end{minipage}
	\caption{Quality of worst element (triangle in 2d or tetrahedron in 3d) after applying optimal deformation for a range of extension factors $\eext$.}
	\label{fig::mesh_quality}
\end{figure}

\subsection{An iterative optimization algorithm}
\label{subsec::iterative}

\begin{algorithm}[h!]
	\caption{Iterative optimization algorithm}
	\label{alg::iterative}
	\begin{algorithmic}[1]
		\Require $ 0 < \atarget \leq \ainit$, $0 < \adec < 1$, $0 < \epsilon$
		\State Set $y_{0}$ to zero
		\State $k \gets 0, \ell \gets 0$
		\State $\alpha_k \gets \ainit$
		\While{$\alpha_k \geq \adec$}
		\Repeat 
		\State Set $y_{\ell}$ as initial guess
		\State Solve \cref{eq::opt_sys_adj_v,eq::opt_sys_adj_p}  for $ \left(v, p \right)_{\ell+1}$
		\State Solve \cref{eq::opt_sys_v,eq::opt_sys_p} 
		for $\left( \mult{v},\mult{p} \right)_{\ell+1}$
		\State \parbox[t]{0.8\textwidth}{ Solve \cref{eq::opt_sys_w,eq::opt_sys_adj_w,eq::opt_sys_b,eq::opt_sys_adj_b,eq::opt_sys_c,eq::opt_sys_adj_vol,eq::opt_sys_adj_bc} for $\left( w, b, \control, \mult{w}, \mult{b}, \mult{\vol}, \mult{\bc} \right)_{\ell+1} $ with semismooth Newton's method and regularization parameter $\alpha_k$\strut}
		\State $\ell \gets \ell+1$
		\Until{$\frac{\Vert \control_{\ell+1} - \control_\ell \Vert_{L^2(\Gobs)}}{\Vert \control_{\ell+1}\Vert_{L^2(\Gobs)}} < \epsilon$}
		\State $\alpha_{k+1} \gets \adec\alpha_{k}$
		\State $k \gets k+1$
		\EndWhile
	\end{algorithmic}
\end{algorithm}

\begin{figure}[h!]
	\centering
	\pgfplotsset{height=0.3\textheight,width=0.9\textwidth,compat=1.9}
	\begin{tikzpicture}
	\begin{axis}[
	title={},
	ymin=1e-7, ymax=2,
	xmin=-1,xmax=53,
	axis y line*=left,
	axis x line=none,
	ylabel={$\alpha$},
	ymode=log,
	]
	\addplot[const plot] table{
	0	1.000000E+00
	1	5.000000E-01
	2	2.500000E-01
	3	1.250000E-01
	4	1.250000E-01
	5	6.250000E-02
	6	6.250000E-02
	7	3.125000E-02
	8	3.125000E-02
	9	1.562500E-02
	10	1.562500E-02
	11	7.812500E-03
	12	7.812500E-03
	13	7.812500E-03
	14	3.906250E-03
	15	3.906250E-03
	16	1.953120E-03
	17	1.953120E-03
	18	9.765620E-04
	19	9.765620E-04
	20	4.882810E-04
	21	4.882810E-04
	22	2.441410E-04
	23	2.441410E-04
	24	2.441410E-04
	25	1.220700E-04
	26	1.220700E-04
	27	6.103520E-05
	28	6.103520E-05
	29	6.103520E-05
	30	3.051760E-05
	31	3.051760E-05
	32	1.525880E-05
	33	1.525880E-05
	34	1.525880E-05
	35	7.629390E-06
	36	7.629390E-06
	37	3.814700E-06
	38	3.814700E-06
	39	3.814700E-06
	40	1.907350E-06
	41	1.907350E-06
	42	1.907350E-06
	43	9.536740E-07
	44	9.536740E-07
	45	9.536740E-07
	46	4.768370E-07
	47	4.768370E-07
	48	4.768370E-07
	49	2.384190E-07
	50	2.384190E-07
	51	2.384190E-07
	52	2.384190E-07
	};
	\label{pgfplots-bars}
	\end{axis}
	\begin{axis}[
	xmin=-1,xmax=53,
	axis y line*=right,
	ylabel={Objective},
	xlabel={\#agglomerated iterations ($\ell$)},
	ymode=log,
	]
	\addplot table{
	0	1.39169
	1	1.3823
	2	1.37549
	3	1.36201
	4	1.33657
	5	1.34222
	6	1.30759
	7	1.31297
	8	1.27663
	9	1.28064
	10	1.24895
	11	1.25149
	12	1.22771
	13	1.23481
	14	1.22705
	15	1.21408
	16	1.21424
	17	1.2049
	18	1.20564
	19	1.20006
	20	1.20084
	21	1.19752
	22	1.19832
	23	1.19619
	24	1.19724
	25	1.19644
	26	1.1958
	27	1.19608
	28	1.1952
	29	1.19584
	30	1.19523
	31	1.19527
	32	1.19511
	33	1.19491
	34	1.19508
	35	1.19477
	36	1.19485
	37	1.19463
	38	1.19463
	39	1.19461
	40	1.19449
	41	1.19449
	42	1.19443
	43	1.19441
	44	1.19432
	45	1.19435
	46	1.19424
	47	1.19427
	48	1.19418
	49	1.19422
	50	1.1941
	51	1.19418
	52	1.19408	
	};
	\label{pgfplots-obj}
	\addlegendentry{Objective}
	\addlegendimage{/pgfplots/refstyle=pgfplots-bars}\addlegendentry{$\alpha$}
	\end{axis}
	\end{tikzpicture}
	\caption{Iterative solution strategy according to \cref{alg::iterative} with $\ainit=1$, $\adec=0.5$ and $\atarget=\num{2e-7}$.}
	\label{fig::iterative-obj}
\end{figure}
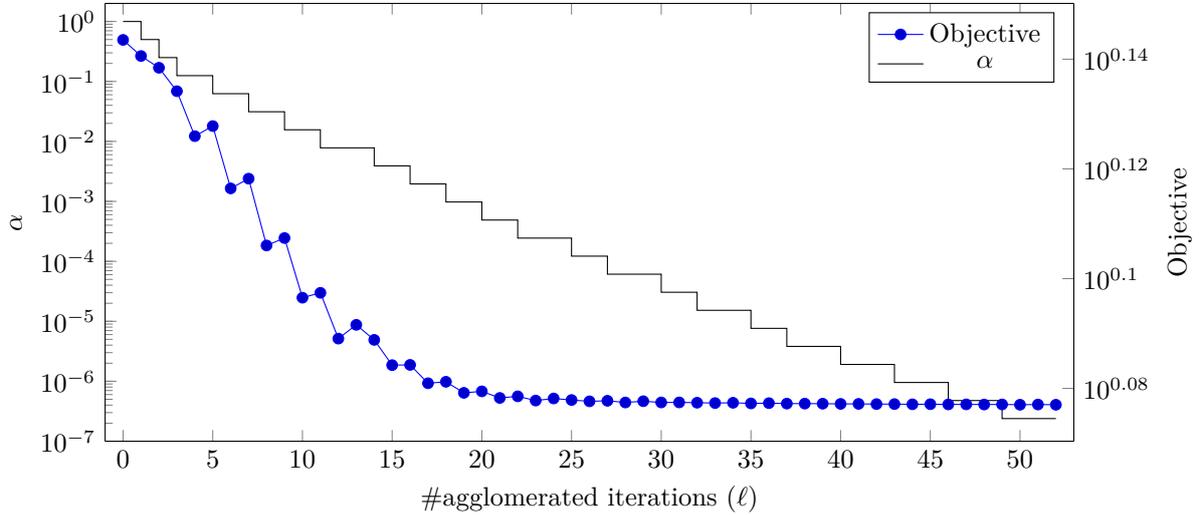
In the previous sections we solve the nonlinear, non-smooth optimality system with the direct solution strategy given in \cref{alg::direct}.
Moreover, a direct solver library is applied to the resulting linear systems within semismooth Newton's method.
This approach is clearly limited due to the high memory requirement.
Especially, when the state equation results from a time-dependent problem, this procedure becomes impracticable.
Hence, in this section we focus on a numerical study of decoupling system \crefrange{eq::opt_sys_w}{eq::opt_sys_adj_bc}.
This approach is summarized in \cref{alg::iterative}.

We demonstrate that it is possible to decouple the solution process of state \cref{eq::opt_sys_adj_v,eq::opt_sys_adj_p}, adjoint \cref{eq::opt_sys_v,eq::opt_sys_p}  and shape related equations, i.e.\ \cref{eq::opt_sys_w,eq::opt_sys_adj_w,eq::opt_sys_b,eq::opt_sys_adj_b,eq::opt_sys_c,eq::opt_sys_adj_vol,eq::opt_sys_adj_bc}, from each other.
On the one hand, this allows to reuse existing solvers for the state equation and embed them into the shape optimization framework.
On the other, the memory requirement for linear solvers significantly reduces.
Moreover, the semismooth part \cref{eq::opt_sys_w} is split from the other equations and a solver can be particularly tailored for this purpose.

\Cref{alg::iterative} operates on the nonlinear optimality system as a fixpoint strategy.
In an outer loop it is again iterated over a decreasing regularization parameter $\alpha$ as in \cref{alg::direct}.
Thus, approximate solutions for the optimization problem according to $\alpha_k$ are utilized as initial guess for the nonlinear solver in iteration $k+1$.
Yet, unlike in the direct approach, the subproblems are only solved approximately by a fixpoint iteration, which solves the decoupled equations of the optimalitiy system in turns.
The termination criterion for this inner loop is the relative change in the control variable $\control$ measured in the $L^2(\Gobs)$-norm.

In \cref{fig::iterative-obj} the results of one run of \cref{alg::iterative} are shown.
The underlying optimization experiment is a 2d computation on the same grid as in \cref{subsec::quality-results} with \num{312} surface segments and \num{6168} triangles in $\Omega$, $\adec = 0.5$, $\ainit = 1.0$, $\atarget=\num{2e-7}$, $\nu = 0.1$ and $\eext = 1.5$.
Note that the initial value of $\alpha$ is significantly larger then the choices made for \cref{alg::direct}.
\Cref{fig::iterative-obj} shows the required inner iterations until the condition
\begin{equation*}
\frac{\Vert \control_{\ell+1} - \control_\ell \Vert_{L^2(\Gobs)}}{\Vert \control_{\ell+1}\Vert_{L^2(\Gobs)}} < \epsilon
\end{equation*}
is fulfilled for $\epsilon = \num{1e-2}$.
Futhermore, the value of the objective $J$ (cf.\ \cref{eq::boundary_control_problem}) is visualized. It is computed in \cref{alg::iterative} in line 10 at the end of one inner loop.
Notice the jumps in the objective function between iteration 5 and 20.
In our experiments it turns out that this is an effect that both influences the minimal possible $\adec$ and $\ainit$.

In this setting a total number of 53 inner iterations, i.e.\ solutions of the state equation, are required to reach the optimal shape.
This numerical study can thus be seen as a proof of concept how to reduce the computational costs of the large, coupled, nonlinear system \crefrange{eq::opt_sys_w}{eq::opt_sys_adj_bc}.
Thus, the proposed method is applicable to more complex problems, such as non-stationary Navier-Stokes flows.

\section{Conclusion}
\label{sec::conclusion}
In this article we have proposed and numerically demonstrated choices of nonlinear extension operators within the method of mappings for aerodynamic shape optimization.
These operators are based on the idea that an additional, nonlinear advection term leads to a rearrangement of discretization cells along the major direction of deformations.

The main goal we have achieved is to circumvent mesh degeneracy effects that appear under large deformations when the extension of the boundary control is chosen according to linear elastic models.
Especially in the underlying aerodynamic drag minimization, where optimal shapes tend to become stretched in flow direction and compressed in the orthogonal directions, we have numerically investigated how mesh quality can be preserved.

We have also demonstrated one possibility to decouple the solution process of the optimality system in order to overcome issues of computational complexity.
Moreover, we have studied how the set of admissible shapes depends on the nonlinearity of the operator and how the local injectivity of mappings can be extended to large deformations.
Since the proposed methodology is formulated in function spaces without taking a specific discretization into account, another benefit of this approach is that it naturally allows to introduce concepts like adaptivity.
An important field for future investigations is a detailed description of properties of the set $\mathcal{F}_\text{adm}$, which is constructed in terms of the nonlinear extension operator $S$.

\section*{Acknowledgment}
The work of the authors has been supported by the Deutsche Forschungsgemeinschaft (DFG, German Research Foundation) within the Research Training Group 2583 \enquote{Modeling, Simulation and Optimization of Fluid Dynamic Applications}.

\printbibliography

\end{document}